\documentclass[12pt,a4paper]{amsart}
\usepackage[]{amsmath, amsthm, amsfonts, amssymb, amscd,mathtools}
\usepackage{enumerate}
\usepackage[mathscr]{eucal}
\usepackage{hyperref}
\usepackage{xypic}
\usepackage[usenames,dvipsnames]{color}
\usepackage{todonotes}

\newcommand{\A}{\mathbb{A}}

\newcommand{\C}{\mathbb{C}}

\newcommand{\N}{\mathbb{N}}

\newcommand{\Q}{\mathbb{Q}}
\newcommand{\R}{\mathbb{R}}
\newcommand{\T}{\mathbb{T}}
\newcommand{\Z}{\mathbb{Z}}

\newcommand{\caE}{\mathcal{E}}

\newcommand{\caO}{\mathcal{O}}

\newcommand{\caS}{\mathcal{S}}

\newcommand{\sJ}{\mathscr{J}}

\newcommand{\fg}{{\mathfrak{g}}}
\newcommand{\ft}{{\mathfrak{t}}}
\newcommand{\fu}{\mathfrak{u}}

\DeclareMathOperator {\gr} {Gr}
\DeclareMathOperator {\rk} {rk}

\DeclareMathOperator{\im}{Im}
\DeclareMathOperator{\Spec}{Spec}

\DeclareMathOperator{\Hom}{Hom}

\DeclareMathOperator{\Res}{Res}
\DeclareMathOperator{\Ind}{Ind}
\DeclareMathOperator{\rec}{rec}
\newcommand{\sat}{\text{\rm sat}}

\newcommand{\bq}{\mathbf{q}}
\newcommand{\dn}[1]{#1^{\prime}}
\newcommand{\dual}[1]{#1^{\vee}}
\newcommand{\face}{\preceq}
\newcommand{\fbd}{\dual{\bar{\varphi}}}

\newcommand{\gl}{\mathfrak{gl}}
\newcommand{\Gm}{\mathbb{G}_{\mathrm{m}}}
\newcommand{\grgs}[3]{\gs{#1}{#2}_{#3}}
\newcommand{\gs}[2]{\Gamma\left(#1,\,#2\right)}

\newcommand{\orth}[1]{#1^{\perp}}
\newcommand{\remvivek}[1]{}

\newcommand{\set}[2]{\left\{#1\,\middle|\,#2\right\}}
\newcommand{\setl}[1]{\left\{#1\right\}}

\DeclareMathOperator{\Ad}{Ad}
\DeclareMathOperator{\ad}{ad}
\DeclareMathOperator{\GL}{GL}
\DeclareMathOperator{\SL}{SL}

\DeclareMathOperator{\cone}{cone}

\numberwithin{equation}{section}

\theoremstyle{plain}
\newtheorem{prop}{Proposition}[section]

\newtheorem{cor}[prop]{Corollary}
\newtheorem{lem}[prop]{Lemma}
\newtheorem{thm}[prop]{Theorem}

\theoremstyle{definition}
\newtheorem{df}[prop]{Definition}

\newtheorem{notation}[prop]{Notation}

\theoremstyle{remark}
\newtheorem{rmk}[prop]{Remark}
\newtheorem{ex}[prop]{Example}

\begin{document}

\title{Multifiltrations}

\author{Jos\'e Ignacio Burgos Gil}

\address{ICMAT\\
Campus Cantoblanco\\
Madrid, Spain}
\email{jiburgos@gmail.com}

\author{Vivek Mohan Mallick}

\address{IISER Pune\\
Dr.~Homi Bhabha Road\\
Pune, India}
\email{vmallick@iiserpune.ac.in}
\date{\today}

\subjclass[2010]{16W70, 14M25, 22E27}

\keywords{Multifiltered vector spaces, toric varieties,
representations of solvable Lie groups}

\begin{abstract}
  This paper deals with properties of filtrations on vector spaces indexed by
  partially ordered finitely generated abelian groups, which we call
  multifiltrations. We discuss the usual properties of
  filtrations, like exhaustivity and separatedness in the setting of
  multifiltrations and we introduce a new notion, regular
  multifiltrations, that has no analogue for classical filtrations. 
  As applications, we provide a reinterpretation of Klyachko, Perling et al's
  description of torus-equivariant torsion-free coherent sheaves on a toric
  variety, and give a simple formula for pull-backs of such sheaves. We also
  explore the occurrence of such multifiltration among finite dimensional
  representations of connected, solvable groups.
\end{abstract}

\maketitle

\section{Introduction}
\label{sec:introduction}

Let $S$ be a totally ordered set (for instance $\Z$ or $\R$) and $E$ a
vector space. An increasing filtration of $E$ is a collection of
subspaces $F_{t}E$, $t\in S$ with the property that, if
$t_{1}\le t_{2}$, then $F_{t_{1}}E\subset F_{t_{2}}E$.

A way to interpret a filtration is thinking of the parameter $t$ as the time
variable and $F_{t}E$ as the information available at time $t$. Only events
that occur before $t$ are accessible at a given time. This interpretation is
asking for a generalization to partially ordered sets. In relativistic
physics, the time is not an absolute coordinate. Given two events $p_1$ and
$p_2$, then $p_{1}$ occurs unconditionally after $p_{2}$ if $p_1-p_2$ is a
positive time-like vector. By contrast, the order of occurrence may
depend on the observer if
$p_1-p_2$ is a space-like vector. In this second case none of the events may affect
the other. Therefore space-time is not a totally ordered set.
In other words, if $S$ is space-time, and $p_{1}, p_{2}\in S$, we
write $p_{1}\ge p_{2}$ if $p_1-p_2$ belongs to the cone of positive
time-like vectors.  This is a partially ordered set.

We generalize the concept of filtration as follows.
If $E$ is a vector
space and $S$ a partially ordered set (poset), an increasing
\emph{multifiltration} indexed by $S$ is a collection of 
subspaces $F_{t}E$, $t\in S$ with the property that, if
$t_{1}\le t_{2}$, then $F_{t_{1}}E\subset F_{t_{2}}E$.

There are several notions related to filtrations with
respect to totally ordered sets. A filtration is
called exhaustive if $E = \bigcup F_{t}E$, that is, if all possible
information is accessible eventually, while a filtration is
called separated if $\{0\}=\bigcap F_{t}E$, that is at time $-\infty$
there is no information. We will see that the notion of exhaustive
filtration carries over to multifiltrations without any change (Definition
\ref{def:2}), while the notion of separated is more subtle and we will
introduce two different notions \emph{chain-separated} (Definition
\ref{def:3}) and \emph{separated} (Definition \ref{def:4}).   

In filtrations with respect to totally ordered sets, a particular
piece of information appears only once and then propagates through the
filtration.
A new phenomenon appears when considering multifiltrations with respect to
posets. The same piece of information may appear in different
unrelated places. Therefore, for multifiltrations the
associated graded vector space may be of bigger dimension that the
original vector space. We introduce the concept of \emph{regular
  multifiltration} to denote a multifiltration in which each piece of information appears
once.  

Multifiltrations with respect to posets may be very wild. In this
paper we
will only consider posets defined using submonoids in finitely generated
torsion free abelian groups.

The concept of multifiltration has appeared in 
\cite{Perling2004} in the context of equivariant coherent torsion free
sheaves on toric varieties, as filtrations indexed by a finitely
generated torsion free group with a preorder induced by a polyhedral
cone.
It also has appeared in \cite{Carlsson2009} in
the context of computational topology, as a filtration indexed by
$\N^{n}$ with its usual partial order. 

In this paper, we give several examples of different pathological and
non-pathological behavior of multifiltrations. 
The main result of the present paper is Theorem \ref{thm:cridsdec} where we
show that a multifiltration is exhaustive separated and regular if and
only if the vector space $E$ has a (non-canonical) multigrading that
defines the multifiltration.

Multifiltrations appear naturally in many situations. The main
application we have in mind is the classification by Klyachko
\cite{Klyachko} of equivariant vector bundles on toric varieties. This
classification was extended to the classification of equivariant
quasi-coherent sheaves on toric varieties and that of equivariant
torsion-free sheaves by Perling \cite{Perling2004}.  In this note we
reinterpret Perling's results on torsion free and locally free
equivariant sheaves in terms of separability and regularity of certain
multifiltrations, see Theorem \ref{thm:1} and Remark \ref{rem:3}.
This interpretation will also allow us to give a concise description of
the pull back of equivariant torsion free sheaves.

Another situation where multifiltrations appear naturally is in the theory of
representations of connected solvable algebraic groups. We will give
several examples of such multifiltrations. 

A completely parallel theory can be developed using
cones in real vector spaces instead of torsion free finitely generated
abelian groups. The difference is similar to
consider, in the classical case, filtrations indexed by $\Z$ or
filtrations indexed by $\R$. 

Since in the applications we
will use decreasing multifiltrations we will focus on them. One can go from
decreasing to increasing multifiltrations reversing the order.

\noindent
\textbf{Acknowledgements:} This work has been developed during a visit
  of the first author to the IISER Pune and a visit of the second
  author to the ICMAT. We are very grateful to both institutions by
  their hospitality. J. I. Burgos was partially supported by MINECO
  research projects MTM2013-42135-P and MTM2016-79400-P and by ICMAT
  Severo Ochoa project SEV-2015-0554. 

\section{Preorders and submonoids on  abelian groups}

In this section we gather some basic facts about preorders and
submonoids on  abelian groups. The main purpose is to fix the notation. 

\begin{df} A \emph{preorder} in a set is a binary relation that is
  transitive and symmetric, but not necessarily antisymmetric. 
  Let $(A, +)$ be an abelian group with identity $0$. We say
  that a preorder $\leq$ on $A$ is \emph{compatible} with the binary
  operation $+$ if
  \begin{equation*}
    a \leq b \iff a + c \leq b + c\ \forall c \in A.
  \end{equation*}
  An abelian group $(A, +)$ along with a compatible preorder
  $\leq$ will be denoted by $(A, +, \leq)$.
\end{df}
\begin{df} Let $(A,+)$ be an abelian group. 
  A \emph{submonoid} $A$ is a subset $C$ such that
  \begin{enumerate}
    \item $0 \in C$, and
    \item $x, y \in C$ implies $x + y \in C$.
    \end{enumerate}
    If $C$ is a submonoid, we write $L(C)=C+(-C)$ for the subgroup of $A$
    generated by $C$.
    A submonoid is called \emph{strict} if and only if $C \cap
    (-C) = \setl{0}$. A submonoid is called \emph{generating} if and only
    if $L(C)=A$.
\end{df}

\begin{lem}\label{lemm:1}
  There is a bijection between the set of all compatible preorders on
  $(A, +)$ and the set of all submonoids in $A$. A preorder is an order if
  and only if the corresponding submonoid is strict and a preorder
  is directed if and only if the submonoid is generating.
\end{lem}
\begin{proof}
  Let $\leq$ be a compatible preorder on $(A, +)$. Then
  \begin{displaymath}
   C_{\geq 0} =
  \set{a \in A}{0 \leq a} 
  \end{displaymath}
  is a submonoid in 
  $(A, +)$. Further, given a submonoid $C$, define $a \leq_{C} a'$ if $a' - a \in
  C$. This defines a preorder. Clearly both constructions are inverse
  of each other.

  Let $\leq$ be a preorder and let $C=C_{\ge
    0}$. The preorder $\le$ is an order if and only
  if the conditions $a' - a \in
  C$ and $a - a' \in
  C$ imply that $a=a'$. Therefore the preorder is an order if and only 
  if the corresponding submonoid is strict.

  Assume that the preorder $\leq$ is directed. Then given any $a \in A$, one
  can find $c \in A$ such that $a \leq c$ and $0 \leq c$. Therefore $c \in
  C$ and $c - a \in C$. Thus,
  \begin{equation*}
    a = c + a - c = c + (- (c-a)) \in C + (-C)
  \end{equation*}
  Since $a$ was arbitrary, $A = C + (-C)$. Therefore $C$ is generating.

  Conversely, assume that $C$ is generating. If $a, b \in A$, $a - b
  \in A = C + (-C)$. Thus one can 
  find elements $c, d \in C$ such that $a - b = c - d$, or $a + d = b + c
  \overset{\text{def}}{=} e$. Since $d \in C$, $a \leq e$, and since
  $c \in C$, $b \leq e$, proving that $\leq$ is directed.
\end{proof}


\begin{notation}
  Let $A$ be an abelian group.
  \begin{enumerate}
  \item If $C\subset A$ is a submonoid, we denote by $\le_C$ the
    corresponding preorder
    \begin{displaymath}
      a\le_{C} b \iff b-a\in C.
    \end{displaymath}
  \item If $\le$ is a compatible preorder, we denote by $C_{\ge 0}$
    the corresponding submonoid
    \begin{displaymath}
      C_{\ge 0}=\set{a\in A}{0\le a}.
    \end{displaymath}

  \item If $\le$ is a compatible preorder we write
    \begin{gather*}
      a<b\quad \iff \quad a\leq b,\ b\not \leq a,\\
      a\lneqq b\quad \iff \quad a\leq b,\ a\not = b.
    \end{gather*}
    We also write $a \ge b$ (respectively, $a > b$, $a \gneqq b$) if
    and only if $b \le a$ 
	(respectively, $b < a$, $b \lneqq a$).
  \end{enumerate}
  In a preorder the symbols $<$ and $\lneqq$ are not equivalent.
\end{notation}

In view of Lemma \ref{lemm:1} we make the following definition.
\begin{df}
  A \emph{preordered abelian group} is a pair $(A,C)$ where $A$ is an
  abelian group and $C\subset A$ is a submonoid.
\end{df}

  Arbitrary submonoids can have very strange shapes and behave in
  pathological ways. We have to add some conditions to the submonoids to
  have nice properties.
  
 \begin{df} Let $(A,C)$ be a preordered abelian group. The saturation
  of $C$, denoted $C_{\sat}$ is the subset
  \begin{displaymath}
    C_{\sat}=\set{x\in A}{\exists N>0,\ N\in \N,\ Nx\in C}.
  \end{displaymath}
  Clearly, $C_{\sat}$ is again a submonoid.
  The submonoid $C$ and the preorder $\le _{C}$ are called
  \emph{saturated} if $C=C_{\sat}$.
\end{df}

\begin{df} \label{def:6}
  An element $q\in A$ is called a \emph{quasi-zero} if $0\le q$ and
  $q\le 0$. The subset of quasi-zero elements is $Z=C\cap (-C)$, so it
  is a subgroup. Let $\bq \colon A \to A / Z$ be the quotient map.
  \label{dfn:quasizer}
\end{df}

The following result is elementary.

\begin{lem} Let $(A,C)$ be a preordered abelian group and
  $Z=C\cap(-C)$  the subgroup of quasi-zero elements. 
  \begin{enumerate}
  \item The preorder on $A$ induces an order on the quotient $A/Z$
    with corresponding strict submonoid $C/Z$.
  \item If $\le$ is saturated, then $A/Z$ is torsion free.
  \end{enumerate}
\end{lem}

\begin{df}\label{def:13}
   A \emph{cone} $C$ in $\R^{n}$ is a convex
  subset such that
  \begin{displaymath}
    x\in C, \lambda \ge 0 \implies \lambda x \in C
  \end{displaymath}
  and a cone of $\Q^{n}$ is the intersection of a cone of $\R^{n}$
  with $\Q^{n}$. 
\end{df}

  \begin{df}
    Let $A$ be a finitely generated (f. g.) abelian group and denote by
    $\iota \colon A \to 
    A\otimes \Q$ the canonical map.
    \begin{enumerate}
    \item   A submonoid $C$ is called \emph{a cone} if it is of the form
      $C=\iota^{-1}(C')$ for a cone $C'$ in $A\otimes \Q$.
    \item A cone is called \emph{strict} (resp. \emph{generating}) if
      the corresponding submonoid is strict (resp. generating).
    \item A cone $C$ (or the corresponding preorder)  is called
      \emph{polyhedral} if it is of the form
      $C=\iota^{-1}(C')$ for a rational polyhedral cone $C'$ in
      $A\otimes \Q$. 
    \end{enumerate}
  We will denote by $\iota _{\R}$ the canonical map $A\to A\otimes
  \R$. 
\end{df}

\begin{lem} \label{lemm:4} Let $A$ be a f.~g.~abelian group  and $C$ a
  submonoid. Then $C_{\sat}$ is a 
  cone and $C$ is a saturated submonoid if and only if it is a cone.
\end{lem}
\begin{proof}
  Let $C'\subset A \otimes \Q$ be the cone of $A \otimes \Q$ generated by
  $C$.  If we see that $C_{\sat}=\iota ^{-1}(C')$, then we deduce that
  $C_{\sat}$ is a 
  cone. Hence if $C=C_{\sat}$, then $C$ is also
  a cone. Conversely, if $C$ is a cone, then $C=\iota^{-1}(C')$,
  therefore $C=C_{\sat}$. Hence we only need to prove that
  $C_{\sat}=\iota ^{-1}(C')$.

  If $y\in \iota ^{-1}(C')$, then we can write
   \begin{displaymath}
    \iota (y)=\sum \frac{n_{i}}{N}\iota(x_{i}),
  \end{displaymath}
  where $N$ and $n_{i}$ are positive integers and $x_{i}\in
  C$. Therefore $\iota(Ny-\sum n_{i}x_{i})=0$. Thus $Ny-\sum
  n_{i}x_{i}$ is torsion and there is an $N'\in \Z_{>0}$ such that
  $N'Ny=\sum N'n_{i}x_{i}\in C$. Therefore $y\in C_{\sat}$.

  Conversely, if $y\in C_{\sat}$, then $Ny\in C$, so $N\iota(y)=\iota
  (Ny)\in C'$. Being $C'$ a cone of $A\otimes \Q$ we deduce that
  $\iota(y)\in C'$. Therefore $y\in \iota ^{-1}(C')$.
\end{proof}

\begin{rmk}
  A submonoid of a f.~g.~abelian group is a polyhedral cone if and
  only if it is saturated and 
  finitely generated.
\end{rmk}

\begin{ex}\label{exm:1}
  The subset of $\Z^{2}$ given by
  \begin{displaymath}
    (x,y)\in C \iff
    \begin{cases}
      x\le 0, \  y\ge 0, &\text{ or}\\
       x > 0, \ y >0, 
    \end{cases}
  \end{displaymath}
  is a cone but is not a polyhedral cone.
\end{ex}

The following result is very useful. It is the analogue of the
existence of ample line bundles in geometry. Intuitively, it says that
any top-dimensional cone has a point in its interior.

\begin{lem}\label{lemm:7}
  Let $(A,C)$ be a preordered f.~g.~abelian group with $C$ a
  generating cone. Then there is an element $a\in C$ such that, for
  all $x\in A$, there is an $n\in \N$ with
  \begin{displaymath}
    x\le na.
  \end{displaymath}
\end{lem}
\begin{proof}
    Since $C$ is generating, it contains a basis
  $\{x_{1},\dots,x_{r}\}$ of the vector space $A\otimes \Q$. Write
  $a=\sum_{i}x_{i}\in C$. Let $x\in A$, then $\iota(x)=\sum
  p_{i}x_{i}$. Choose $n\ge 0$ such that $n\ge p_{i}$,
  $i=1,\dots,r$. Since $C$ is a cone, $C=\iota ^{-1}(C')$ for a cone
  $C'\subset A\otimes \Q$. Then $na-x\in A$ and
  \begin{displaymath}
    \iota(na-x)\in \sum_{i=1}^{r}\Q_{\ge 0}x_{i}\subset C'.
  \end{displaymath}
  Therefore $na-x\in C$ proving the result. 
 \end{proof}

\begin{lem} \label{lemm:6}
  Let $A$ be a f.~g.~abelian group. If a strict cone $C\subset A$ exists,
  then $A$ is torsion free.
\end{lem}
\begin{proof}
  Let $C\subset A$ be a strict cone and $T\subset A$ the subgroup of
  torsion elements. Since $\iota(T)=0$ and $C$ is a cone, we deduce
  that $T\subset C$. Since $T=-T$, then $T\subset (-C)\cap C$. Since
  $C$ is strict, $T=0$.
\end{proof}

We define strongly strict submonoids here which will play a role later (in
Lemma~\ref{lemm:2}).

\begin{df} \label{dfn:strspcon}
  A submonoid $C $ of $A$ is called \emph{strongly strict} if
  there exist a closed cone $D\subset A\otimes \R$ with $D\cap
  (-D)=\setl{0}$ and $C\subset \iota _{\R}^{-1}(D)$.  
\end{df}

The interest of strongly strict cones comes from the following
separation result.

\begin{lem} \label{lemm:5}
  Let $D\subset \R^{n}$ be a closed cone such that $D\cap
  (-D)=\setl{0}$. Then there is a linear form $\omega \colon \R^{n}\to
  \R$ such that $D\cap \ker(\omega )=\setl{0}$  and $\omega (x)\ge 0$
  for all $x\in D$. In particular $\omega (x)>0$ for all $0\not = x
  \in D$.
\end{lem}
\begin{proof}
  Choose an Euclidean norm in $\R^{n}$ and let $S^{n-1}$ be the unit
  sphere in $\R^{n}$. Write $E$ for the convex hull of $(-D)\cap
  S^{n-1}$. Recall that our definition of cone includes being
  convex (cf. Definition \ref{def:13}). Thus $E$ and $D$ are closed
  convex subsets. We check that 
  they have empty intersection. Indeed if $x\in E\cap D$, since
  $E\subset -D$ and $(-D)\cap D=0$, we deduce that $x=0$. If $0\in E$,
  by the definition of $E$, we can write
  \begin{displaymath}
    0=\sum_{i=1}^{k}\alpha _{i}x_{i},
  \end{displaymath}
  with $k\ge 2$, $\sum \alpha _{i}=1$, and for  $i=1,\dots,k$, $\alpha
  _{i}>0$, , $x_{i}\in -D$ and 
  $\|x_{i}\|=1$. Then
  \begin{displaymath}
    -D \ni x_{1} = -\sum_{i=2}^{k} \frac{\alpha _{i}}{\alpha
      _{1}}x_{i}\in D,
  \end{displaymath}
  which implies that $x_{1}=0$, contradicting that $\|x_{1}\|=1$.

  Moreover, $E$ being compact, $E$ and $D$ do not have any
  recession direction in common. By \cite[Corollary
  11.4.1]{Rockafellar} they are strongly separated. $D$ being a cone,
  by \cite[Theorem 11.7]{Rockafellar} they can be strongly separated
  by an hyperplane through the origin. Statement that is equivalent to
  the existence of $\omega $ as in the lemma.
\end{proof}

\begin{cor}\label{cor:1} Let $A$ be a f. g. abelian group, 
  $C$ a strongly strict submonoid, $Z\subset C$ the subgroup of
  quasi-zero elements, $D\subset A\otimes \R$ a cone as in
  Definition \ref{dfn:strspcon} and $\omega $ the linear form
  provided by Lemma \ref{lemm:5}. Then
  \begin{displaymath}
    \inf \set{\omega (\iota_{\R}(x))}{x\in C\setminus Z} > 0.
  \end{displaymath}
\end{cor}
\begin{proof}
  Choose an Euclidean norm in $A\otimes \R$.
  The elements $\iota _{\R}(A)$ form a lattice of $A\otimes
  \R$. Therefore there is a constant $K_{1}>0$ such that $\|x\|\ge
  K_{1}$ for any $0\not = x \in \iota _{\R}(A)$. By the construction
  of $\omega $, there is a constant $K_{2}>0$ such that, for all $x\in
  D$, we have $\omega (x)\ge K_{2}\|x\|$. Combining these two
  inequalities we deduce that $\omega (x) \ge K_{1}K_{2}$ for all $0\not
  = x\in\iota_{R}(A)\cap D$.

  If $x\in C\setminus Z$, then $\iota_{\R}(x)\not =0$ and
  $\iota_{\R}(x)\in \iota_{R}(A)\cap D$ from which the result follows. 
\end{proof}

\begin{ex}
  Let $C'$ be a strictly convex polyhedral cone of $A\otimes \Q$. Then
  $\iota ^{-1}(C')$ is a strongly strict cone. 
\end{ex}

\begin{ex}  \label{exm:ststrict}
  The cone in Example \ref{exm:1}, is strict but not strongly
  strict. In the other direction,  if a f. g. abelian group $A$ has
  torsion, then strongly strict submonoids  
  of $A$ do not need to be strict cones. If $A$ is
  torsion free, then any strongly strict submonoid is strict. 
\end{ex}

\begin{rmk}
  When dealing with infinitely generated abelian groups, like real
  vector spaces, it is useful to have a topology and work with partial
  orders that are compatible with the topology. Due to the applications
  we have in mind,  we will restrict ourselves to finitely generated
  abelian groups.
\end{rmk}

\section{Multifiltrations}
\label{sec:poset-filtration}

In this section we study some basic properties of
multifiltrations indexed by a finitely generated preordered abelian
group. For simplicity of the exposition we will restrict ourselves to
multifiltrations on finite dimensional vector spaces, but most of the
definitions and results carry over to more general situations.

In this section, by a vector space we mean a finite dimensional
vector space.  Let $E$ be such a  vector space. 

\begin{df}\label{def:7}
  A decreasing multifiltration $F$ on $E$ is the data of a
  f.~g.~preordered abelian group $(A,C)$ and a collection of
  subobjects $\{F^{\lambda }E\}_{\lambda \in A}$,
  $F^{\lambda }E\subset E$, such that
  \begin{displaymath}
    \lambda \le \mu \Longrightarrow F^{\lambda }E\supset F^{\mu }E.
  \end{displaymath}
  The preordered abelian group $(A,C)$ is called the \emph{index set}
  of the multifiltration.

  There is the analogous definition of an increasing multifiltration
  $\{F_{\lambda } E\}_{\lambda \in A}$. 
\end{df}

\begin{rmk}
  If we fix the set of indices $(A,C)$, then the class of finite
  dimensional vector spaces provided with a multifiltration with index
  set $(A,C)$ is a category. 
\end{rmk}

We can induce multifiltrations on subspaces and quotients.

\begin{df}\label{def:8}
  With $(A,C)$ a preordered abelian group, $E$ a vector space,
  $S\subset E$ a vector subspace and $Q=E/S$ the quotient, let $F$ be
  a multifiltration on 
  $E$ indexed by $(A,C)$. The \emph{subspace multifiltration} and the
  \emph{quotient multifiltrations}, both denoted also by $F$ are given
  by 
  \begin{align*}
    F^{\lambda }S&= S\cap F^{\lambda } E,\\
    F^{\lambda} Q&= \frac{F^{\lambda }E}{F^{\lambda }E \cap S}
  \end{align*}
  respectively.
\end{df}

Multifiltrations can be shifted.
\begin{df}\label{def:9}
    With $(A,C)$ a preordered abelian group, $E$ a vector space,
    let $F$ be a multifiltration on $E$ indexed by $(A,C)$ and
    $\lambda \in A$. The shifted multifiltration $F[\lambda ]$ is given
    by
    \begin{displaymath}
      F[\lambda ]^{\mu }E=F^{\lambda +\mu }E.
    \end{displaymath}
  \end{df}

  \begin{rmk}
    The shift of multifiltrations is an equivalence of categories. 
  \end{rmk}

We can define two functors that allow us to change the set of
indices. 

\begin{df}\label{def:1}
Let $\varphi\colon A\to A'$ be a morphisms of abelian groups and $C,C'$
submonoids in $A$ and $A'$ respectively such that $\varphi(C)\subset
C'$. Then the map $\varphi$ is order preserving. In this case we say
that $\varphi\colon (A,C)\to (A',C')$ is a morphism of preordered
abelian groups. From $\varphi$ we can define two
functors on multifiltered vector spaces.
\begin{enumerate}
\item Let $F$ be a multifiltration on
  $E$ indexed by $(A',C')$. Then the restricted multifiltration,
  $\Res^{\varphi}(F)$ is 
  defined as
  \begin{displaymath}
    \Res^{\varphi}(F)^{\lambda }E=F^{\varphi(\lambda )}E, \ \lambda \in A.
  \end{displaymath}
\item Let $F$ be a multifiltration on
  $E$ indexed by $(A,C)$. Then the induced multifiltration
  $\Ind_{\varphi}(F)$ is
  defined as
  \begin{displaymath}
    \Ind_{\varphi}(F)^{\lambda}E=\sum_{\mu \colon \lambda \underset{C'}{\le}
      \varphi(\mu )}F^{\mu }E, \ \lambda \in A'.
  \end{displaymath}
\end{enumerate}
\end{df}

The proof of the next result about adjointness of $\Ind$ and $\Res$ is
elementary.

\begin{prop} Let $\varphi \colon (A,C)\to (A',C')$ be a morphism of
  preordered groups. Then 
  the functors $\Ind_{\varphi}$ and $\Res^{\varphi}$ form an adjoint
  pair. That is, for every pair of multifiltered vector spaces $(E,F)$
  and $(E',F')$ with index sets $(A,C)$ and $(A',C')$,
  \begin{displaymath}
    \Hom\big((E,\Ind_{\varphi}(F)),(E',F')\big)=\Hom\big((E,F),(E',\Res^{\varphi}(F'))\big).  
  \end{displaymath}
\end{prop}

\begin{rmk}\label{rem:1}
  Combining the restriction and induction of representations of
  Definition \ref{def:1} with the shift of Definition \ref{def:9}, we
  can define restriction and induction of representations with respect
  to ``affine'' maps of the form $\lambda \mapsto \varphi(\lambda
  )+\mu $, where $\varphi\colon (A,C)\to (A',C')$ is a morphism of
  preordered abelian groups and $\mu \in A$.
\end{rmk}

The next result follow easily from the definitions.
\begin{prop}\label{prop:7}
  Let $\varphi\colon A\to A'$ and $\psi \colon A'\to A''$ be morphisms
  of abelian groups and $C$, $C'$ and $C''$ 
submonoids in $A$, $A'$ and $A''$ respectively such that $\varphi(C)\subset
C'$ and $\psi (C')\subset C''$. Then
\begin{displaymath}
  \Ind _{\psi \circ \varphi} = \Ind_{\psi }\circ \Ind_{\varphi}
  \quad \text{ and }\quad
  \Res^{\psi \circ \varphi}=\Res^{\varphi}\circ \Res^{\psi }.
\end{displaymath}
\end{prop}

In the next result we show that using restriction and induction of
filtrations we can always reduce to the case when $C$ is strict and
therefore the associated preorder is an order.

\begin{prop}\label{prop:6}
  Let $(A,C)$ be a preordered abelian group, and $Z$ the subgroup of
  quasi-zero elements of $C$.  Write $A'=A/Z$, $C'=C/Z$ and
  $\bq\colon A\to A'$ the quotient map. Let $E$ be a finite
  dimensional vector space.
  \begin{enumerate}
  \item \label{item:23} Let $F$ be a multifiltration on $E$ indexed by
  $(A,C)$, then
  \begin{displaymath}
    F=\Res^{\bq}(\Ind_{\bq }(F)).
  \end{displaymath}
\item \label{item:24} Let $F$ be a multifiltration on $E$ indexed by
  $(A',C')$, then
  \begin{displaymath}
    F=\Ind_{\bq }(\Res^{\bq}(F)).
  \end{displaymath}
  \end{enumerate}
\end{prop}
\begin{proof}
  We start proving \eqref{item:23}. For all
  $\lambda \in A$ and $q\in Z$ we have that  $\lambda +q\le \lambda $
  and $\lambda \le \lambda +q$. Therefore for all $\lambda ,\mu \in
  A$, $\lambda \le \mu $ if and only if $\bq(\lambda )\le \bq(\mu
  )$. Let $F$ be a multifiltration indexed by $(A,C)$ then 
  $F^{\lambda }E=F^{\lambda
    +q}E$. Thus
  \begin{multline*}
    \Res^{\bq}(\Ind_{\bq }(F))^{\lambda }E
    =\Ind_{\bq }(F)^{\bq(\lambda) }E\\
    =\sum_{\mu \colon \bq(\lambda )\le\bq(\mu)}F^{\mu }E
    =\sum_{\mu \colon \lambda \le\mu}F^{\mu }E=F^{\lambda }E.
  \end{multline*}

  Conversely, let $F$ be a multifiltration indexed by $(A',C')$. Then
  \begin{multline*}
    \Ind_{\bq }(\Res^{\bq}(F))^{\lambda }E
    =\sum_{\mu \colon \lambda \le\bq(\mu)}\Res^{\bq}(F)^{\mu }E
    =\sum_{\mu \colon \lambda \le\bq(\mu)}F^{\bq(\mu) }E
    =F^{\lambda }E,
  \end{multline*}
  proving \eqref{item:24}. 
\end{proof}

We want to translate the classical notions of exhaustivity and
separability of filtrations  to multifiltrations. Exhaustivity is easy.
\begin{df} \label{def:2}
  A multifiltration $F$ on $E$ is \emph{exhaustive} if
  \begin{displaymath}
    E=\sum _{\lambda \in A}F^{\lambda }E.
  \end{displaymath}
\end{df}

\begin{prop}\label{prop:4}
  Let $\varphi\colon (A,C)\to  (A',C')$ be a morphism of preordered
f.~g.~abelian groups. 
\begin{enumerate}
\item \label{item:13} Let $(E,F)$ be a multifiltered vector space with
  index set $(A,C)$ and $\lambda \in A$. If $F$ is exhaustive, then
  $F[\lambda ]$ is  exhaustive. 
\item \label{item:11} Let $(E,F)$ be a multifiltered vector space with
  index set $(A,C)$. If $F$ is exhaustive, then $\Ind_{\varphi}(F)$ is
  exhaustive. 
\item \label{item:12} Let $(E,F)$ be a multifiltered vector
  space with index set $(A',C')$. If $F$ is 
  exhaustive and $\varphi$ is surjective, then $\Res^{\varphi}(F)$ is
  exhaustive. 
\end{enumerate}
\end{prop}
\begin{proof}
  Statement \eqref{item:13} is clear.
  Assume that $F$ is an exhaustive multifiltration indexed by
  $(A,C)$. Then 
  \begin{displaymath}
    \sum_{\lambda \in A'}\Ind_{\varphi}(F)^{\lambda }E=
    \sum_{\lambda \in A'}\sum_{\mu \colon \varphi(\mu ) \ge \lambda }F^{\mu
    }E=
    \sum_{\mu \in A}F^{\mu }E=E.
  \end{displaymath}
  This proves \eqref{item:11}.

  Assume that $F$ is an exhaustive multifiltration indexed by
  $(A',C')$, and that $\varphi$ is surjective. Then 
  \begin{displaymath}
    \sum_{\lambda \in A}\Res^{\varphi}(F)^{\lambda }E=
    \sum_{\lambda \in A}F^{\varphi(\lambda )}E\overset{\ast}{=} 
    \sum_{\mu \in A'}F^{\mu }E=E,
  \end{displaymath}
  where the equality $\overset{\ast}{=}$ holds by the surjectivity of
  $\varphi$. This proves \eqref{item:12}.
\end{proof}

\begin{ex}
  Clearly, the surjectivity of $\varphi$ is needed in Proposition
  \ref{prop:4}~\eqref{item:12} as the following example shows. Let
  \begin{align*}
    A&=\Z,&C&=\set{m}{m\ge 0},\\
    A'&=\Z^{2},&C'&=\set{(m,n)}{m\ge 0,\, n\ge 0 },
  \end{align*}
  and $\varphi\colon A\to A'$ the map $m\mapsto (m,0)$. Let $E$ be a one
  dimensional vector space and $F$ the multifiltration indexed by
  $(A',C')$ defined as
  \begin{displaymath}
    F^{(m,n)}E=
    \begin{cases}
      E,&\text{ if }n\le 1, m\le 1,\\
      \setl{0},&\text{ otherwise}.
    \end{cases}
  \end{displaymath}
  This is an exhaustive multifiltration but
  \begin{displaymath}
    \Res^{\varphi}(F)^{m}E=\setl{0},\ \forall m\in \Z,
  \end{displaymath}
  hence is not exhaustive.
\end{ex}

Separability is a little more subtle. The naive approach would be to
impose separability on chains.
\begin{df} \label{def:3}
  A multifiltration $F$ is called \emph{chain-separated} if, for any
  infinite sequence
  \begin{displaymath}
    \lambda _{1}< \lambda _{2}<\dots < \lambda _{i}<\dots 
  \end{displaymath}
  there is an index $i_{0}$ such that, for all $i\ge i_{0}$,
  $F^{\lambda _{i}}E=0$. 
\end{df}

The notion of chain-separatedness does not change if we take the
quotient by the quasi-zero elements. The following result follows
immediately from the definitions. 

\begin{prop} \label{prop:5}
  Let $(A,C)$ be a preordered abelian group, $Z$ be the subgroup of
  quasi-zero elements of $C$.  Write $A'=A/Z$, $C'=C/Z$ and
  $\bq\colon A\to A'$ the quotient map. Let $E$ be a finite
  dimensional vector space and $F$ a multifiltration on $E$ indexed by
  $(A,C)$. Then $F$ is chain-separated if and only if $\Ind_{\bq}F$ is
  chain separated.
\end{prop}

The notion of chain-separatedness is not satisfactory because, in
general, it is not
stable under induced multifiltrations as the next example shows.
\begin{ex} Let $A=\Z^{2}$, $C=\set{(x,y)}{x,y\ge 0}$ and $E$ a one
  dimensional vector space. Then the multifiltration
  \begin{displaymath}
    F^{(x,y)}E=
    \begin{cases}
      E,&\text{ if }x+y\le 0,\\
       0,&\text{ if }x+y> 0.
    \end{cases}
  \end{displaymath}
 is chain-separated. But there is an infinite sequence of
 non-comparable points with $F^{\lambda }E$ different from zero. If we now
 consider the preordered group $(A',C')$ with $A'=\Z$ and $C'$ the cone of non-negative
 integers, and the map $\varphi\colon A\to A'$ being the first
 projection $(x,y)\mapsto x$. Then $\Ind_{\varphi}(F)$ is not
chain-separated.
\end{ex}

We need a stronger definition of separatedness, that we will give
after introducing some notation. 

\begin{notation} Let $(E,F)$ be a multifiltered vector space with
  index set $(A,C)$. 
  For a set $K \subset A$ we denote
  \begin{equation*}
    \widetilde{K} = \set{\lambda \in A}{\lambda \geq \mu \text{ for some } \mu
  \in K}
\end{equation*}
and
\begin{displaymath}
  F^{K}E=\sum_{\lambda \in K}F^{\lambda }E.
\end{displaymath}
Clearly $F^{K}E=F^{\widetilde K}E$.
\end{notation}

  \begin{df} \label{def:4}
  A multifiltration indexed by $(A,C)$ is called \emph{separated} if for all collections
  $K_{\alpha}$ of subsets of $A$ satisfying $\cap_{\alpha}
  \widetilde{K_{\alpha}} = \emptyset$, we have
  \begin{equation*}
    \bigcap F^{K_{\alpha}}E = \setl{0}.
  \end{equation*}
\end{df}

The following lemma implies that in many cases, separatedness
implies chain-separatedness.

\begin{lem}\label{lemm:2}
  Let $A$ be a f.~g.~abelian group,  $C$ a submonoid and $Z\subset C$ the
  subgroup of quasi-zero elements. Assume that  $C/Z$ is a strongly
  strict submonoid of  $A/Z$. Let $(E,
  F)$ be a multifiltration indexed by $(A,C)$.  If $F$ is separated, then it
  is chain-separated. 
\end{lem}
\begin{proof}
  Let $\lambda _{1}<\lambda _{2}<\dots$ be an infinite sequence of strict
  inequalities. We shall show that
  \begin{equation}
    \bigcap _{i} \{\lambda _{i} \}^{\sim} = \emptyset.
	\label{eqn:evennoth}
  \end{equation}
  Write $A' = A /Z$, and $C' = C / Z$. Recall that $\bq : A \to A'$
  denotes the quotient
  map. If $\lambda_1 < \lambda_2 < \dotsb$ in the order induced by
  $C$, then 
  $\bq(\lambda_1) < \bq(\lambda_2) < \dotsb$ in the order induced by $C'$.
  Thus it is enough to show that
  \begin{displaymath}
 \bigcap_{i\ge 1} \set{a' \in
  A'}{a'\,\geq_{C'}\,\bq(\lambda_i)} = \emptyset.
  \end{displaymath}

Since $C'$ is a strongly strict submonoid, there is a closed cone 
  $D\subset A' \otimes \R$ with $D\cap (-D)=\setl{0}$ and $C'\subset
  \iota'_{\R}{}^{-1}(D)$.

  Let $\omega $ be the linear form provided by Lemma
  \ref{lemm:5}. Since $\bq(\lambda_1) < \bq(\lambda_2) < \dotsb$, by
  Corollary \ref{cor:1}, the sequence
  \begin{displaymath}
    \omega (\iota_{\R}(\bq(\lambda_1))) <
    \omega (\iota_{\R}(\bq(\lambda_2))) <\dotsb
  \end{displaymath}
  is not bounded above.

  Since any element in $y\in \bigcap_i
  \setl{\bq(\lambda_i)}^{\sim}$ should satisfy
  \begin{displaymath}
      \omega (\iota_{\R}(y))\ge   \omega (\iota_{\R}(\bq(\lambda_i)))
  \end{displaymath}
  for all $i$, we deduce $\bigcap_i
  \setl{\lambda_i}^{\sim} = \emptyset$, proving \eqref{eqn:evennoth}.

  From equation \eqref{eqn:evennoth}, since $F$ separated, we deduce
  $\bigcap_{i}F^{\lambda _{i}}E=\setl{0}$. 
\end{proof}

\begin{ex}
  We give an example of a multifiltration which is separated but not
  chain-separated. Consider the cone $C$ in Example \ref{exm:1} (see also
  Example \ref{exm:ststrict}). Let $E$ be a one dimensional vector space
  over a base field $k$ and let
  \[
	F^{(m, n)} E = 
	\begin{cases}
	  E & \text{if } (m, n) \notin C \\
	  0 & \text{if } (m, n) \in C
	\end{cases}
  \]
  We shall first show that $F$ is separated. 
  Note that the cone $C$ satisfies $A \setminus (-C) \subset C$. Thus, for
  a collection of subsets $K_{\alpha} \subset A$,
  \begin{align*}
	\bigcap_{\alpha} \widetilde{K_{\alpha}} = \emptyset
	&\implies \forall a \in A, \exists \alpha, a \notin \widetilde{K_{\alpha}}
	\implies \exists \alpha, 0 \notin \widetilde{K_{\alpha}}. \\
	&\implies \exists \alpha, \forall \lambda \in K_{\alpha}, 0 \not\geq \lambda
	\implies \exists \alpha, \forall \lambda \in K_{\alpha}, \lambda \notin
	(-C) \\
	&\implies \exists \alpha, \forall \lambda \in K_{\alpha}, \lambda \in C
	\implies \exists \alpha, F^{K_{\alpha}} E = \setl{0} \\
	&\implies \bigcap_{\alpha} F^{K_{\alpha}} E = \setl{0}.
  \end{align*}
  
  However, the multifiltration is not chain-separated, as for $\lambda_n =
  (-n, 0)$ for $n = 1, 2, \dotsc$, $\lambda_1 < \lambda_2 < \dotsb$.
  However, note that $F^{\lambda_n} E = E$ for all $n \geq 1$, and hence
  $\bigcap_{n} F^{\lambda_n} E \neq \setl{0}$.
\end{ex}

We next derive some properties of the notion of separatedness.

\begin{prop}\label{prop:1}
Let $\varphi\colon (A,C)\to  (A',C')$ be a morphism of preordered
f.~g.~abelian groups.
\begin{enumerate}
\item \label{item:14} Let $(E,F)$ be a multifiltered vector space with
  index set $(A,C)$ and $\lambda \in A$. If $F$ is 
separated, then $F[\lambda ]$ is separated.
\item \label{item:4} Let $(E,F)$ be a multifiltered vector space with
  index set $(A,C)$. If $F$ is 
separated, then $\Ind_{\varphi}(F)$ is separated.
\item \label{item:5} Let $(E,F)$ be a multifiltered vector
  space with index set $(A',C')$. If $F$ is 
  separated and one of the following two sets of conditions is
  satisfied:
  \begin{enumerate}
  \item $\varphi$ is surjective and $C=\varphi^{-1}(C')$ or
  \item \label{item:6} $C'$ is a polyhedral cone,  $C=\varphi^{-1}(C')$ and
    $L(C)=\varphi^{-1}(L(C'))$,
  \end{enumerate}
   then $\Res^{\varphi}(F)$ is separated.
\end{enumerate}
\end{prop}
\begin{proof}
  Statement \ref{item:14} is obvious. 
  We prove \eqref{item:4}.
  Let $K_{\alpha }'$ be a family of subsets of $A'$ such that
  $\bigcap_{\alpha }
  \widetilde {K}'_{\alpha }=\emptyset$. Write
  \begin{displaymath}
   K_{\alpha }=\set{\lambda }{\varphi(\lambda )\in
     \widetilde{K}'_{\alpha }}. 
  \end{displaymath}
  Then $\widetilde K_{\alpha }=K_{\alpha }$. The condition $\bigcap_{\alpha }
  \widetilde {K}'_{\alpha }=\emptyset$ implies the condition $\bigcap_{\alpha }
  \widetilde {K}_{\alpha }=\emptyset$. Moreover
  \begin{displaymath}
    \Ind_{\varphi}(F)^{\widetilde {K}'_{\alpha }}E=\sum_{\lambda' \in
      \widetilde {K}'_{\alpha }} \Ind_{\varphi}(F)^{\lambda' }E=\sum_{\varphi(\lambda )\in
      \widetilde {K}'_{\alpha }} F^{\lambda }E=\sum_{\lambda \in
      \widetilde {K}_{\alpha }} F^{\lambda }E=F^{\widetilde {K}_{\alpha }}E.
  \end{displaymath}
  Since the original multifiltration was separated,
  \begin{displaymath}
    \bigcap \Ind_{\varphi}(F)^{\widetilde {K}'_{\alpha }}E=\bigcap
    F^{\widetilde {K}_{\alpha }}E=0.  
  \end{displaymath}

  We next prove \eqref{item:5} in the case when $\varphi$ is
  surjective and $C=\varphi^{-1}(C')$. 
  Let $\setl{K_{\alpha }}$ be a collection of subsets such that
  $\bigcap \widetilde{K}_{\alpha }=\emptyset$. Then $\bigcap
  \varphi(K_{\alpha })^{\sim}=\emptyset$. Assume that this is not the case
  and let $b\in \bigcap \varphi(K_{\alpha })^{\sim}$. Then, for each
  $\alpha $ there is a $c_{\alpha }\in K_{\alpha }$ with
  $b-\varphi(c_{\alpha })\in C'$. Since $\varphi$ is surjective,
  $b=\varphi(a)$. Since $C=\varphi^{-1}(C')$ we deduce $a-c_{\alpha
  }\in C$. Therefore $a\in \bigcap \widetilde{K}_{\alpha }$ which
  gives us a contradiction.

  Using that $\Res^{\varphi}(F)^{K_{\alpha }}E=F^{\varphi(K_{\alpha
    })}E$, and that $F$ is separated,  we compute
  \begin{displaymath}
    \bigcap_{\alpha } \Res^{\varphi}(F)^{K_{\alpha }}E=
	\bigcap_{\alpha } F^{\varphi(K_{\alpha })}E = \setl{0}.
   \end{displaymath}
   showing that $\Res^{\varphi}(F)$ is separated.

   Finally we treat the case when $C'$ is polyhedral,
   $C=\varphi^{-1}(C')$ and $L(C)=\varphi^{-1}(L(C')))$. As in the
   previous case the key is to show that, if $\setl{K_{\alpha }}$ is a
   collection of subsets such that 
  $\bigcap \widetilde{K}_{\alpha }=\emptyset$. Then $\bigcap
  \varphi(K_{\alpha })^{\sim}=\emptyset$. Assume that this is not
  the case and let $b\in \bigcap
  \varphi(K_{\alpha })^{\sim}$. This means that, for each $\alpha $,
  there is a $c_{\alpha
  }\in K_{\alpha }$ with $b-\varphi(c_{\alpha })\in C'$. Thus
  $c_{\alpha }\in \varphi^{-1}(b-C')$. Fix $\alpha _{0}$. For every
  other $\alpha $ we have that $\varphi(c_{\alpha }-c_{\alpha
    _{0}})\in (b-C')-(b-C')=L(C')$. Therefore
  \begin{displaymath}
    c_{\alpha }-c_{\alpha _{0}}\in \varphi^{-1}(L(C'))=L(C).
  \end{displaymath}
  Thus
  \begin{displaymath}
    c_{\alpha }\in \varphi^{-1}(b-C')\cap (c_{\alpha _{0}}+L(C)).
  \end{displaymath}
  Since $C'$ is polyhedral, the subset $\varphi^{-1}(b-C')\cap
  (c_{\alpha _{0}}+L(C))$ is a polyhedral subset in the sense that it
  is of the form $\iota ^{-1}(K)$ for a rational polyhedral subset
  $K\subset A\otimes \Q$. Moreover, $C=\iota ^{-1}(C_{\Q})$ for a
  polyhedral cone $C_{\Q}$ in $A\otimes \Q$. Being $K$ polyhedral, we
  know that $K=\Delta +\rec(K)$, where $\Delta$ is a polytope and
  $\rec(K)$ is the recession set of $K$ (see \cite[\S 8]{Rockafellar}
  for the definition and properties of the recession cone of a convex
  set) by  \cite[Corollary 8.3.4]{Rockafellar} and the fact that
  $\varphi^{-1}(C')=C$ we have that $\rec(K)=-C_{\Q}$. Since $\iota^{-1}(\Delta )$
  is a bounded subset of $c_{\alpha _{0}}+L(C)$, there exists $a\in A$
  such that $\iota^{-1}(\Delta )\subset a-C$. We conclude that $\iota
  ^{-1}(K)\subset a-C$. Therefore, for each $\alpha $, $a-c_{\alpha
  }\in C$ and $a\in \bigcap \widetilde{K}_{\alpha }$ contradicting the
  hypothesis. The rest of the proof in this case follows the same
  pattern as the previous case.
\end{proof}

\begin{ex}
  We give some examples showing that all the hypothesis in Proposition
  \ref{prop:1}~\eqref{item:5} are needed. Let $E$ be a vector space of
  dimension $1$.
  \begin{enumerate}
  \item Let $\varphi\colon \Z^{2}\to \Z$ be the map $(m,n)\mapsto
    m$. Consider the cones
    \begin{align*}
      C'&=\set{m\in \Z}{m\ge 0},\\
      C&=\set{(m,n)\in \Z^{2}}{m\ge 0,n\ge 0}.
    \end{align*}
    Then $C\not
    =\varphi^{-1}(C')$. Let $F$ be the multifiltration of $E$ indexed by
    $(\Z,C')$ given by $F^{m}E=E$ if and only if $m\le 
    0$. This is a separated multifiltration. But the multifiltration
    \begin{displaymath}
      \Res^{\varphi}(F)^{(m,n)}E=
      \begin{cases}
        E,&\text{ if }m\le 0,\\
        0,&\text{ if }m> 0.
      \end{cases}
    \end{displaymath}
    is not
    chain-separated  with respect to $C$ as shown by the infinite sequence
    \begin{displaymath}
      \dots < (0,i) < (0,i+1) < \dots
    \end{displaymath}
    Note that, with respect to $\varphi^{-1}(C')$ the strict
    inequalities in this sequence become $\le$.
  \item Let $\varphi\colon \Z \to \Z^{2}$ given by
    $\varphi(m)=(m,-m)$. Consider the cone
    \begin{displaymath}
            C'=\set{(m,n)\in \Z^{2}}{m\ge 0,n\ge 0},
    \end{displaymath}
    and write $C=\varphi^{-1}(C')=\setl{0}$. In this case $C'$ is
    polyhedral but $\varphi^{-1}(L(C'))\not = L(C)$.
    Let $F$ be the multifiltration of $E$
    indexed by $(\Z^{2},C')$ given by
    \begin{displaymath}
      F^{(m,n)}E=E \iff (m,n)\le (1,1).
    \end{displaymath}
    Then the multifiltration $\Res^{\varphi}(F)$ is given by
    \begin{displaymath}
      \Res^{\varphi}(F)^{m} E=E \iff m=-1,0,1.
    \end{displaymath}
    This multifiltration is not separable. Consider the sets
    $K_{0}=\setl{0}$ and $K_{1}=\setl{1}$. Then $\widetilde
    K_{0}=K_{0}$ and $\widetilde K_{1}=K_{1}$ because
    $C=\setl{0}$. Therefore $\widetilde
    K_{0}\cap \widetilde K_{1}=\emptyset$ but
    \begin{displaymath}
      \Res^{\varphi}(F)^{0} E \cap \Res^{\varphi}(F)^{1} E=E\not = \emptyset. 
    \end{displaymath}
  \item Let $\varphi\colon \Z^{2}\to \Z^{3}$ be the map $(m,n)\mapsto
    (m,n,0)$. Let $K_{1}\subset \R^{3}$ be the convex subset given by
    \begin{displaymath}
      (x,y,z)\in K_{1}\iff z=1, \ y\ge |x|-\log(1+|x|)+1
    \end{displaymath}
    and let $K\subset \R^{3}$ be closure of the cone generated by
    $K_{1}$. Write $C'=K\cap \Z^{3}$. This is a non-polyhedral
    cone. Write $C=\varphi^{-1}(C')$. Then
    \begin{displaymath}
      C=\rec(K_{1})\cap \Z^{2}=\set{(m,n)}{n\ge |m|}.
    \end{displaymath}
    In this case $\varphi^{-1}(L(C'))=L(C)$. So all hypothesis are
    satisfied except the fact that $C'$ is polyhedral.     Let $F$ be
    the multifiltration of $E$ 
    indexed by $(\Z^{3},C')$ given by
    \begin{displaymath}
      F^{(m,n,\ell)}E=E \iff (m,n,\ell)\le (0,0,1).
    \end{displaymath}
    This multifiltration is easily seen to be separable.
    The multifiltration $\Res^{\varphi}(F)$ is given by
    \begin{displaymath}
      \Res^{\varphi}(F)^{(m,n)} E=E \iff n\le -|m|+\log(1+|m|)-1.
    \end{displaymath}
    We claim that this multifiltration is not separable. Indeed, let
    $S\subset \Z^{2}$ be the set
    \begin{displaymath}
      S=\set{(m,n)\in \Z^{2}}{n\le  -|m|+\log(1+|m|)-1}.
    \end{displaymath}
    For every $(m,n)\in S$ write $K_{(m,n)}=\setl{(m,n)}$. Then we
    claim that
    \begin{displaymath}
      \bigcap_{(m,n)\in S}\widetilde K_{(m,n)}=\emptyset.
    \end{displaymath}
    Assume that this is not the case and let $(a,b)\in \bigcap\widetilde
    K_{(m,n)}$. This implies that, for all $(m,n)\in S$,
    $(a,b)-(m,n)\in C$. Choose $m<0$ such that $a-m>0$ and that
    \begin{displaymath}
      \log(1-m)>b-a+2.
    \end{displaymath}
    write $n=b-a+m+1$.
    Then
    \begin{displaymath}
      n<m+\log(1-m)-1=-|m|+\log(1+|m|)-1
    \end{displaymath}
    so $(m,n)\in S$. But
    \begin{displaymath}
      b-n=a-m-1<a-m=|a-m|.
    \end{displaymath}
    therefore $(a,b)-(m,n)\not \in C$ contradicting the
    assumption. Since
    \begin{displaymath}
      \bigcap_{(m,n)\in S} \Res^{\varphi}(F)^{(m,n)} E=E\not = \setl{0}
    \end{displaymath}
    we conclude that $\Res^{\varphi}(F)$ is not separable.
  \end{enumerate}
\end{ex}

The next property of multifiltrations that we discuss is
\emph{regularity}. This property has no analogue for
filtrations. Roughly speaking, regularity tells us whether the information described by
a multifiltration appears once or it appears in several places independently.

\begin{df}\label{def:5}
  Let $(A, C)$ be a preordered abelian group and $(E, F)$ a
  multifiltration indexed by $(A, C)$. The multifiltration $F$ is called
  \emph{regular} if for every pair of subsets $K_{1}, K_{2}\subset A$,
  the equality
  \begin{displaymath}
    \sum_{\lambda  \in K_1} F^{\lambda} E \cap \sum_{\mu \in K_2}F^{\mu} E =
    \sum_{\mu' \in \widetilde{K}_1 \cap \widetilde{K}_2} F^{\mu'} E
  \end{displaymath}
  holds.  
\end{df}

\begin{prop}\label{prop:3}
  Let $\varphi\colon (A,C)\to  (A',C')$ be a morphism of preordered
  f.~g.~abelian groups.
  \begin{enumerate}
\item \label{item:15} Let $(E,F)$ be a multifiltered vector space with
  index set $(A,C)$ and $\lambda \in A$. If $F$ is regular, then
  $F[\lambda ]$ is regular.
  \item \label{item:7}   Let $(E,F)$ be a multifiltered vector space with
  index set $(A,C)$. If $F$ is regular, then $\Ind_{\varphi}(F)$ is
  regular.
\item \label{item:8} Assume that $\varphi$ is surjective and that
  $C=\varphi^{-1}(C')$. Let $(E,F)$ be a multifiltered vector
  space with index set $(A',C')$. If $F$ is 
  regular, then $\Res^{\varphi}(F)$ is regular. 
  \end{enumerate}
\end{prop}
\begin{proof} Assertion \eqref{item:15} is clear.
  We start proving \eqref{item:7}.
  Let $K'_{1},K'_{2}\subset A'$. Write $K_{i}=\varphi^{-1}(\widetilde
  {K'_{i}})$, $i=1,2$. They satisfy $\widetilde K_{i}=K_{i}$. Then
  \begin{align*}
     \sum_{\lambda  \in K'_1} \Ind_{\varphi}(F)^{\lambda} E \cap
    \sum_{\mu \in K'_2}\Ind_{\varphi}(F)^{\mu} E
    & =
       \sum_{\lambda\colon \varphi(\lambda )  \in \widetilde {K'_1}}
      F^{\lambda} E
      \cap \sum_{\mu\colon \varphi(\mu ) \in \widetilde {K'_2}}F^{\mu} E\\ 
    &=
       \sum_{\lambda  \in K_1} F^{\lambda} E \cap \sum_{\mu \in
      K_2}F^{\mu} E\\
    &\overset{\ast}{=}
      \sum_{\mu' \in \widetilde{K}_1 \cap \widetilde{K}_2} F^{\mu'}
      E\\
    & =\sum_{\mu'\colon \varphi(\mu ') \in \widetilde{K'_1} \cap
      \widetilde{K'_2}} F^{\mu'} E\\
    &= \sum_{\mu' \in \widetilde{K'_1} \cap \widetilde{K'_2}} \Ind_{\varphi}(F)^{\mu'} E,
  \end{align*}
  where the equality $\overset{\ast}{=}$ follows from the regularity
  of $F$.

  We next prove \eqref{item:8}. The key point is to prove that, when
  $\varphi$ is surjective and  $C=\varphi^{-1}(C')$, then for any pair
  of sets $K_{1},K_{2}\subset A$, the equality
  \begin{equation}\label{eq:2}
    \varphi(\widetilde K_{1}\cap \widetilde
    K_{2})^{\sim}=\varphi(K_{1})^{\sim}\cap\varphi(K_{2})^{\sim} 
  \end{equation}
 holds. The inclusion $\subseteq$ is clear: If $x\in
 \varphi(\widetilde K_{1}\cap \widetilde K_{2})^{\sim}$ means that
 $x\ge \varphi(y)$ for some $y\in\widetilde K_{1}\cap \widetilde K_{2}
 $. This implies that there are $\lambda \in K_{1}$, $\mu \in K_{2}$
 with $y\ge \lambda $ and $y\ge \mu $. Therefore $x\ge \varphi(\lambda
 )$ and $x\ge \varphi(\mu )$ which tells us that $x\in
 \varphi(K_{1})^{\sim}\cap\varphi(K_{2})^{\sim} $. For this inclusion
 we have not used any property of the function $\varphi$. 

 We see now the reverse inclusion where the properties of $\varphi$
 will be needed.  Assume that $x\in
 \varphi(K_{1})^{\sim}\cap\varphi(K_{2})^{\sim} $. This means that
 there are $\lambda \in K_{1}$, $\mu \in K_{2}$, $c'_{1},c'_{2}\in C'$
 such that
 \begin{displaymath}
   x=\varphi(\lambda )+c_{1},\quad x=\varphi(\mu )+c_{2}.
 \end{displaymath}
 Since $\varphi$ is surjective we can write $x=\varphi(y)$, Moreover,
 $C=\varphi^{-1}(C')$, in particular $C'=\varphi(C)$, so we can write
 $c'_{1}=\varphi(c_{1})$, $c'_{2}=\varphi(c_{2})$ for $c_{1},c_{2}\in
 C$. Then
 \begin{displaymath}
   \varphi(y)=\varphi(\lambda )+\varphi(c_{1}),\quad
   \varphi(y)=\varphi(\mu  )+\varphi(c_{2}).
 \end{displaymath}
The condition $C=\varphi^{-1}(C')$ together with $\varphi(y-\lambda
-c_{1})=0\in C'$ implies that $y\ge \lambda +c_{1}\ge \lambda
$. Similarly $y\ge \mu $. So $y\in \widetilde K_{1}\cap \widetilde
K_{2}$ and
\begin{displaymath}
  x\in \varphi( \widetilde K_{1}\cap \widetilde K_{2})\subset \varphi(
  \widetilde K_{1}\cap \widetilde K_{2})^{\sim},
\end{displaymath}
proving equation \eqref{eq:2}.

We now compute
  \begin{align*}
     \sum_{\lambda  \in K_1} \Res^{\varphi}(F)^{\lambda} E \cap
    \sum_{\mu \in K_2}\Res^{\varphi}(F)^{\mu} E
    & =
       \sum_{\lambda\in K_1}
	   F^{\varphi(\lambda)} E
      \cap \sum_{\mu\in K_2} F^{\varphi(\mu)} E\\ 
    &=
       \sum_{\lambda  \in \varphi(K_1)} F^{\lambda} E \cap \sum_{\mu \in
     \varphi (K_2)}F^{\mu} E\\
    &\overset{\ast}{=}
      \sum_{\mu' \in \varphi(K_1)^{\sim} \cap \varphi(K_2)^{\sim}} F^{\mu'}
      E\\
    &\overset{\ast\ast}{=}
      \sum_{\mu' \in \varphi(\widetilde K_1 \cap\widetilde K_2)^{\sim}} F^{\mu'}
      E\\
    & =
      \sum_{\mu' \in \varphi(\widetilde K_1 \cap\widetilde K_2)} F^{\mu'}
      E\\
    &=
     \sum_{\mu' \in \widetilde K_1 \cap\widetilde K_2} \Res^{\varphi}(F)^{\mu'},
  \end{align*}
where the equality $\overset{\ast}{=}$ follows from the regularity of
$F$ and the equality $\overset{**}{=}$ from equation \eqref{eq:2}.
\end{proof}

\begin{ex}\label{exm:2}
  The condition of being regular is not stable under
  restriction of multifiltrations.
  \begin{enumerate}
  \item \label{item:9} We first show that the condition of
    $\varphi$ being surjective is not enough. 
  Let $A=\Z^{2}$, $C=\set{(m, n)}{m\ge 0,
    n\ge 0}$, $A'=\Z$, $C'=\set{m}{m\ge 0}$ and let $\varphi\colon
  (A,C)\to (A',C')$ be the map $\varphi(m,n)=m+n$. Let $E$ be a one
  dimensional vector space and let $F$ be the multifiltration indexed by
  $(A',C')$ given by
  \begin{displaymath}
    F^{m}E=
    \begin{cases}
      E,&\text{ if }m\le 0\\
      \setl{0}, &\text{ if }m> 0.
    \end{cases}
  \end{displaymath}
  This is a regular multifiltration. The restricted multifiltration
  $\Res^{\varphi}(F)$ is the multifiltration
  \begin{displaymath}
    \Res^{\varphi}(F)^{(m,n)}=
    \begin{cases}
      E,&\text{ if }m+n\le 0\\
      \setl{0}, &\text{ if }m+n> 0.
    \end{cases}
  \end{displaymath}
      which is not regular. In fact, if $\lambda =(0,0)$ and
      $\mu =(1,-1)$, then
      \begin{displaymath}
        \Res^{\varphi}(F)^{\lambda} E \cap \Res^{\varphi}(F)^{\mu} E=E,
      \end{displaymath}
      but $\{\lambda \}^{\sim}\cap \{\mu \}^{\sim}=\{(1,0)\}^{\sim}$. Thus
      \begin{displaymath}
        \sum_{\mu '\in \{\lambda \}^{\sim}\cap \{\mu
          \}^{\sim}}\Res^{\varphi}(F)^{\mu'}E=\Res^{\varphi}(F)^{(1,0)}E=\setl{0}. 
      \end{displaymath}
    \item \label{item:10} We also give an example showing that, if
      $\varphi$ is not surjective, even if we ask the hypothesis of
      Proposition \ref{prop:1}~\eqref{item:6}, the multifiltration
      $\Res^{\varphi}(F)$ may fail to be regular. Let
      \begin{align*}
        \label{eq:3}
        A&=\Z^{2},&C=&\cone((1,1),(-1,1))\\
        A'&=\Z^{3},&C'=&\cone((1,1,1),(-1,1,1),(1,1,-1),(-1,1,-1)),
      \end{align*}
      and let $\varphi\colon A\to A'$ be the map
      $\varphi(n,m)=(n,m,0)$. Then $C' $ is polyhedral,
      $C=\varphi^{-1}(C')$ and $L(C)=\varphi(L(C'))$.

      Again let $E$ be a one dimensional vector space and consider the
      multifiltration indexed by $(A',C')$
      \begin{displaymath}
        F^{(\ell,m,n)}E=
        \begin{cases}
          E,&\text{ if }(\ell,m,n)\le (1,1,1),\\
          \setl{0},&\text{ otherwise}.
        \end{cases}
      \end{displaymath}
      This multifiltration is regular. By contrast the multifiltration
      $\Res_{\varphi}(F)$ is not regular.  Indeed consider the points
      $\lambda =(0,0)$ and $\mu =(2,0)$. Since
      \begin{align*}
        (1,1,1)-\varphi(0,0)&=(1,1,1)\in C',\text{ and }\\
        (1,1,1)-\varphi(2,0)&=(-1,1,1)\in C',
      \end{align*}
      we deduce that
      \begin{displaymath}
        \Res^{\varphi}(F)^{(0,0)}E\cap \Res^{\varphi}(F)^{(2,0)}E=E.
      \end{displaymath}
      On the other hand
      \begin{displaymath}
        \setl{(0,0)}^{\sim} \cap \setl{(2,0)}^{\sim}=\setl{(1,1)}^{\sim}
      \end{displaymath}
      Since $(1,1,1)-\varphi(1,1)=(0,0,1)\not \in C',$
      we deduce that
      \begin{displaymath}
        \sum_{\mu '\in \setl{(0,0)}^{\sim}\cap
          \setl{(2,0)}^{\sim}}\Res^{\varphi}(F)^{\mu'}E=F^{(1,1,0)}E=\setl{0}.          
      \end{displaymath}

  \end{enumerate}
\end{ex}

\begin{prop}
  Let $(A,  C)$ be a preordered f.~g.~abelian group. Let $E$ be a vector 
  space with a multifiltration $F$. Let $S\subset E$ be a vector
  subspace and $Q=E/S$ the quotient.
  \begin{enumerate}
  \item If $C$ is generating and $F$ is exhaustive, the
    multifiltrations induced in $S$ and $Q$ 
    are exhaustive.
  \item If $F$ is separated the
    multifiltration induced in $S$ is separated. If moreover $C$ is
    generating, the multifiltration induced in $Q$ is separated. 
  \item If $F$ is chain-separated, the
    multifiltrations induced in $S$ and $Q$ are 
    chain-separated.
  \end{enumerate}
  \label{lem:sepexhds}
\end{prop}
\begin{proof}
  Assume that $F$ is exhaustive and let $x\in S$. The multifiltration on
  $E$ being exhaustive, there is a finite subset $\lambda _{i}$,
  $i=1,\dots,r$,
  \begin{displaymath}
    x\in \sum_{i=1}^{r} F^{\lambda_{i} }E.
  \end{displaymath}
  Since $C$ is generating, there is a $\mu \in A$ such that $\mu \le
  \lambda _{i}$, $i=1,\dots,r$. Then $x\in F^{\mu }E\cap S$ showing
  that the multifiltration induced in $S$ is exhaustive. 
  Let now $y\in Q$ and choose $x\in E$ a representative. As before,
  there is a $\mu \in A$ such that $x\in F^{\mu }E$. Therefore $y\in
  F^{\mu }Q$.

  We next assume that the multifiltration on $E$ is separated and we see
  that the multifiltration induced in $S$ is 
  separated. Let $\setl{K_{\alpha }}$ be a family of subsets such that
  $\bigcap \widetilde K_{\alpha }=\emptyset$. Then
  \begin{displaymath}
    F^{K_{\alpha }}S=\sum_{\lambda \in K_{\alpha }}F^{\lambda
    }S \subset F^{K_{\alpha }}E.
  \end{displaymath}
  Therefore, using the separatedness of the multifiltration on $E$,
  \begin{displaymath}
    \bigcap_{\alpha }F^{K_{\alpha }}S\subset \bigcap_{\alpha }F^{K_{\alpha
	}}E=\setl{0}.
  \end{displaymath}

  We next prove that the multifiltration induced in $Q$ is
  separated using that $C$ is generating. Let $\setl{K_{\alpha }}$ be
  as before and let $a\in C$ be 
  the element provided by Lemma \ref{lemm:7}. Since $\bigcap
  \widetilde K_{\alpha }=\emptyset$, for every $n\ge 0$ there is an
  $\alpha _{n}$ such that $na\not \in \widetilde K_{\alpha_{n}
  }$. Write
  \begin{displaymath}
    J_{m}=\bigcup_{n\ge m}K_{\alpha _{n}}.
  \end{displaymath}
  Then
  \begin{equation}
    \label{eq:4}
    \widetilde J_{1}\supset \widetilde J_{2}\supset
  \dots .
  \end{equation}
  Moreover, $\bigcap \widetilde J_{m}=\emptyset$. Assume this
  is not the case. Then there exists an $x\in \bigcap \widetilde
  J_{m}$. Since there is an $n\ge 0$ with $x\le na$, we deduce that
  $na\in \bigcap \widetilde J_{m}$ which contradicts the construction
  of the $J_{m}$. By separatedness of $F$, we know that $\bigcap
  F^{J_{m}}E=0$, and by \eqref{eq:4}, we have
  $F^{J_{1}}E\supset F^{J_{2}}E\supset\dots .$
Therefore, since $E$ is finite dimensional, there is an $m_{0}$ with
$F^{J_{m_{0}}}E=0$.  Then
\begin{displaymath}
  \bigcap_{\alpha } F^{K_{\alpha }}Q\subset
  \bigcap_{n } F^{K_{\alpha _{n}}}Q\subset
  \bigcap_{m } F^{J_{m}}Q\subset
  F^{J_{m_{0}}}Q=0.
\end{displaymath}
So, the multifiltration induced in $Q$ is separated.

  Assume now that the multifiltration $F$ is chain-separated and let
  \begin{displaymath}
    \lambda _{1}<\lambda _{2}<\dots 
  \end{displaymath}
  an infinite sequence of strict inequalities. Since we are assuming that $E$
  is finite dimensional, there is an $i$ such that $F^{\lambda
    _{i}}E=0$. Therefore
  \begin{displaymath}
    F^{\lambda _{i}}S=F^{\lambda _{i}}E \cap
      S=0,\text{ and }
    F^{\lambda _{i}}Q=\frac{F^{\lambda_{i} }E}{F^{\lambda _{i}}E \cap
      S}=0,
  \end{displaymath}
  showing that the multifiltrations induced in $S$ and  $Q$ are chain-separated. 
\end{proof}

\begin{ex} \label{exm:4}
  Regularity is not preserved under taking subspaces and quotients.
  \begin{enumerate}
  \item \label{item:16} Let $A=\Z^{2}$ and $C=\set{(m,n)}{m\ge 0,n\ge
      0}$. Let $E $ be a two dimensional vector space and let
    $\setl{u,v}$ be a basis of $E$. Consider the multifiltration of $E$
    \begin{displaymath}
      F^{(n,m)}E=
      \begin{cases}
        E,&\text{ if }m\le 0,\ n\le 0,\\
        \langle u \rangle,& \text{ if }m\le 0,\ n=1,\\
        \langle v \rangle,& \text{ if }m =1,\ n\le 0,\\
        0,&\text{ otherwise.}
      \end{cases}
    \end{displaymath}
    This multifiltration is regular. Let $S=\langle u+v \rangle \subset E$
    be a one dimensional subspace and let $Q=E/S$ be the quotient. The
    multifiltration induced in the quotient is not regular because
    \begin{displaymath}
      F^{(1,0)}Q=F^{(0,1)}Q=Q,
    \end{displaymath}
    while $F^{(1,1)}Q=0$.
  \item \label{item:17} Let $A=\Z^{2}$ again, but let $C$ be the
    submonoid generated by $(0,1),(1,1),(2,1)$. Then $C$ is a
    generating cone. As before, let $E $ be a two dimensional vector
    space with basis $\setl{u,v}$. Consider the multifiltration of $E$
    \begin{displaymath}
      F^{(n,m)}E=
      \begin{cases}
        E,&\text{ if }m< 0,\ 2n \le m,\\
        \langle u \rangle,& \text{ if }m= 0,\ n\le 0,\\
        \langle v \rangle,& \text{ if }m < 0,\ 2n=m+1,\\
        0,&\text{ otherwise.}
      \end{cases}
    \end{displaymath}
    This is a regular multifiltration, as can be described as
    \begin{align*}
      u\in F^{(n,m)}E&\Longleftrightarrow (m,n)\le (0,0),\\
      v\in F^{(n,m)}E&\Longleftrightarrow (m,n)\le (-1,0).
    \end{align*}
    Nevertheless, the multifiltration induced in $S=\langle u+v\rangle$ is
    the multifiltration
    \begin{displaymath}
      F^{(n,m)}S=S \Longleftrightarrow  m<0,\ 2n\le m,
    \end{displaymath}
    which is not regular because
    \begin{displaymath}
      F^{(-2,-1)}S=F^{(-1,-1)}S=S,\text{ while }
      F^{(-1,0)}S=F^{(0,0)}S=0.
    \end{displaymath}
  \item \label{item:18} Let $A=\Z^{3}$ and $C=\set{(\ell,m,n)}{\ell
      \ge0, m\ge 0,n\ge
      0}$. Let $E $ be a 3-dimensional vector space and let
    $\setl{u,v,w}$ be a basis of $E$. Consider the multifiltration of
    $E$ defined by
    \begin{align*}
      u\in F^{(\ell,n,m)}E&\Longleftrightarrow (\ell,m,n)\le (-1,0,0),\\
      v\in F^{(\ell,n,m)}E&\Longleftrightarrow (\ell,m,n)\le
                            (0,-1,0),\\
      w\in F^{(\ell,n,m)}E&\Longleftrightarrow (\ell,m,n)\le (0,0,-1).
    \end{align*}
    This multifiltration is regular. Let $S\subset E$ be the two
    dimensional subspace generated by $\setl{u-v,v-w}$. The multifiltration
    induced in $S$ is not regular as can be seen easily considering
    $K_{1}=\setl{(-1,-1,0),(-1,0,-1)}$ and $K_{2}=\setl{(0,-1,-1)}$.
  \end{enumerate}
\end{ex}

\begin{rmk} \label{rem:2}
Observe that the examples \ref{exm:4}~\eqref{item:17} and
\ref{exm:4}~\eqref{item:18} are more
involved than Example \ref{exm:4}~\eqref{item:16}. 
Recall that, if $A$ is a torsion free f.~g.~abelian
  group, a generating cone $C\subset A$ is called unimodular if
  $C$ is generated by an integral basis of $A$    
The cones of  \ref{exm:4}~\eqref{item:16} and
\ref{exm:4}~\eqref{item:18} are unimodular,
while the one of \ref{exm:4}~\eqref{item:17} is not. On the other hand, in
Example \ref{exm:4}~\eqref{item:18}, the group $A$ has rank 3, while
the group of the other two examples has rank two.  As we will see in
Section \ref{sec:equiv-sheav-toric} the examples
\ref{exm:4}~\eqref{item:17} and \ref{exm:4}~\eqref{item:18} are
related to reflexive sheaves on toric varieties that are not locally
free. 
\end{rmk}

As in the case of filtrations, to any vector space with a
multifiltration we can associate a graded vector space.

\begin{df} Let $(A,C)$ be a preordered abelian group and $(E,F)$ a
  vector space with a multifiltration indexed by $(A,C)$. For $\lambda
  \in A$, we define
  the graded piece of degree $\lambda $ as
  \begin{displaymath}
    \gr_{F}^{\lambda }E=F^{\lambda }E\left/ \sum_{\mu >\lambda} F^{\mu}E\right. 
  \end{displaymath}
  The \emph{associated graded vector space} is
  \begin{displaymath}
    \gr_{F}^{\ast}E=\bigoplus_{\lambda \in A} \gr_{F}^{\lambda }E. 
  \end{displaymath}  
  A point $\lambda \in A$ is said to be a \emph{jump point} if
  $\gr_{F}^{\lambda }E\not = 0$.
\end{df}

In general, the graded vector space associated to a multifiltered
vector space $E$ can
have a different dimension from that of $E$. Even it may happen that
$E$ is finite dimensional and $\gr_{F}^{\ast}E$ is infinite dimensional.
To study the associated graded vector space we start by studying the
basic properties of the jump points.

\begin{lem} Let $(E,F)$ be a multifiltered vector space with index set
  $(A,C)$ and let $Q$ be a quotient of $E$. If
  $\lambda_0$ is a jump point of $(Q,F)$, then $\lambda_0$ is a
  jump point of $(E,F)$.
  \label{lem:jpptlift}
\end{lem}

\begin{proof}  Let $\lambda _{0}$ be a jump point of $(Q,F)$, so
  $\sum_{\mu > \lambda_0} F^{\mu}Q \subsetneqq F^{\lambda_0}Q$. Choose
  $f \in F^{\lambda_0}Q \setminus 
  \sum_{\mu > \lambda_0} F^{\mu}Q$. Let $e \in F^{\lambda_0}E$ be such that
  its image in $F^{\lambda_0}Q$ is $f$. Clearly, $e \notin \sum_{\mu >
  \lambda_0} F^{\mu }E$. Thus $\lambda_0$ is a jump point of the multifiltration on
  $E$.
\end{proof}

\begin{lem}
  Suppose $(A, C)$ is a preordered f.~g.~abelian group, $E$  a
  finite dimensional vector space, and $F$ a decreasing, exhaustive and
  chain-separated multifiltration. Let $\sJ \subset A$ be the set of
  jump points of $F$. Then,
  \begin{equation*}
    E = \sum_{j \in \sJ} F^{j}E.
  \end{equation*}
  \label{lem:jumpssum}
\end{lem}
\begin{proof}
  We begin by demonstrating that if $E \neq
  \setl{0}$, then $\sJ \neq \emptyset$.

  Suppose $\dim E = 1$ and that the multifiltration on $E$ has no jump points.
  Since the multifiltration is exhaustive, there exists $\lambda_0$ such that
  $F^{\lambda_0} E= E$. As $\lambda_0$ is not a jump point $F^{\lambda_0} E=
  \sum _{\mu > \lambda_0} F^{\mu}E$. Thus, as $\dim E = 1$, there exists
  $\lambda_1 > \lambda_0$ such that $F^{\lambda_1}E = E$. Continuing, we get
  an increasing sequence
  \begin{equation*}
    \lambda_0 < \lambda_1 < \lambda_2 < \dotsb
  \end{equation*}
  such that $F^{\lambda_n} E= E$ for all $n \geq 0$. Therefore $\cap_n
  F^{\lambda_n} E= E$ which contradicts the fact that the 
  multifiltration is chain-separated. We deduce that the multifiltration admits a jump point.

  Now suppose $E$ is general and $S \subset E$ is a vector subspace such
  that $\dim E/S = 1$. Let $Q = E/S$. The induced multifiltration on $Q$ is
  chain-separated and exhaustive by Lemma \ref{lem:sepexhds}. Since $Q$ has
  dimension 1, its multifiltration admits a jump point. So, by Lemma
  \ref{lem:jpptlift}, the multifiltration on $E$ has a jump point,
  proving the claim.

  We now prove the lemma. Write $V = \sum_{\lambda \in \sJ}
  F^{\lambda}E$, and suppose that $V \neq E$. Set
  $U = E/V$. Since $U \neq \setl{0}$, the induced multifiltration admits a jump
  point say at $\mu_0$. Then by Lemma \ref{lem:jpptlift}, $\mu_0$ is a jump
  point of $E$ not contained in $\sJ$, which is a contradiction. Thus
  $E = \sum_{\lambda \in \sJ} F^{\lambda}E$.
\end{proof}

For the proof of the following lemma, we make a few definitions.

\begin{df}
  A subset $S \subset A$ is said to be \emph{bounded below} if there exists
  $\underline{a} \in A$ such that $\underline{a} \leq \lambda$ for all
  $\lambda \in S$. A subset $S \subset A$ is said to be \emph{bounded
  above}, if there exists an element $\overline{a} \in A$ such that $\lambda
  \leq \overline{a}$ for all $\lambda \in S$. $S$ is said to be
  \emph{bounded} if it is both bounded below and bounded above.
\end{df}

\begin{rmk}
  Suppose $A$ is a finitely generated abelian group and $C$ is a strict
  submonoid. If a subset $S \subset A$ is a bounded set, then it is finite.
\end{rmk}

\begin{lem}
  \label{lemm:3}
  Let $(A,C)$ be a preordered f.~g.~abelian group with $C$ a generating
  cone. Let $(E,F)$ be a finite dimensional vector space
  with $F$ separated. Then the set of jump points is bounded above. 
\end{lem}
\begin{proof}
  Let $a$ be the element provided by Lemma \ref{lemm:7}.
 Define 
  \begin{equation*}
    K_n = \set{\lambda \in A}{\lambda \nleq na},\  n\in \N.
  \end{equation*}
  Note that $K_n = \widetilde{K_n}$ and
  \begin{equation*}
    \bigcap_{n\in \N} \widetilde{K_n} = \bigcap K_n = \set{\lambda \in A}{\lambda
    \nleq na\ \forall n \in \N} = \emptyset,
  \end{equation*}
  by the property defining $a$. Thus by separatedness,
  \begin{equation*}
    \bigcap F^{K_n} E = \setl{0}.
  \end{equation*}
  On the other hand, $F^{K_n} E$ is a decreasing filtration:
  \begin{equation*}
    F^{K_1} E \supset F^{K_2} E \supset F^{K_3} E \supset \dotsb
  \end{equation*}
  Since $E$ is finite dimensional, $\cap_n F^{K_n} E = F^{K_{n_0}} E$
  for some $n_0 \in \N$. Thus, $F^{K_{n_0}} E = \setl{0}$. Now note that if
  $\lambda \in \sJ$, $F^{\lambda} E \neq \setl{0}$. As $F^{K_{n_0}} E =
  \sum_{\lambda \nleq n_0 a} F^{\lambda} E = \setl{0}$, $\lambda \in \sJ$
  implies that $\lambda \leq n_0 a$ and hence $n_0 a$ is an upper bound for
  $\sJ$. 
\end{proof}

When the submonoid  $C$ is a polyhedral strict cone, there is an even stronger
result. 

\begin{prop}\label{prop:2}
  Let $(A,C)$ be a preordered f.~g.~abelian group with $C$ a polyhedral
  generating strict cone. Let $(E,F)$ be a finite dimensional vector space
  with $F$ separated. Then the set of jump points is finite. 
\end{prop}
\begin{proof}
  By Lemma \ref{lemm:3}, the set of jump points is bounded above. We
  prove the result by induction on the rank of $A$. If $\rk(A)=0$,
  then $A$ is finite and we are done. Assume that $\rk(A)=r>0$.

  If, for every $a \in A$ there is a $\mu \in \sJ$ with $\mu
  <a $, we can construct an infinite sequence of jump points
  \begin{displaymath}
    \mu _{1}> \mu _{2}>\dots
  \end{displaymath}
  Which implies that $F^{\mu _{1}}E\subsetneqq F^{\mu
    _{2}}E\subsetneqq \dots$, contradicting the fact that $E$ is
  finite dimensional. Thus, there are two elements $a,b\in A$ such
  that all the jump points $\lambda \in \sJ$  satisfy
  \begin{equation}
    \label{eq:1}
    \lambda \le a,\qquad \lambda \not < b.
  \end{equation}

  Since $C$ is a polyhedral cone, there exists a finite number of linear maps
  $m_{i}\in (A\otimes \Q)^{\vee}$, $i=1,\dots,r$, such that
  \begin{displaymath}
    C=\set{\lambda \in A}{m_{i}(\iota(\lambda ))\ge 0, \forall i, 1 \leq i
  \leq r},
  \end{displaymath}
  where $\iota\colon A\to A\otimes \Q$ is the canonical map, and the
  $m_{i}$ are the equations describing the facets of $C$. Since
  $\iota (A)$ is a full lattice of $A\otimes \Q$, we can assume that
  the $m_{i}$ have integral values on $A$. Denote
  \begin{displaymath}
    A_{i,n}=\set{\lambda \in A}{m_{i}(\iota(\lambda ))=n}.
  \end{displaymath}
  This is a torsor over the abelian subgroup $A_{i,0}$.

  By the condition \eqref{eq:1}, we deduce
  \begin{displaymath}
    \sJ\subset \bigcup_{i=1}^{r} \ \bigcup_{m_{i}(b)\le n \le m_{i}(a)}
    A_{i,n}. 
  \end{displaymath}
  This is a finite union of torsors.

  Denote $C_{i}=C \cap A_{i,0}$. For each $i,n$, choose an element
  $\lambda _{i,n}\in A_{i,n}$ and define a multifiltration $F_{i,n}$
  indexed by $(A_{i,0},C_{i})$, given by 
  \begin{displaymath}
    F_{i,n}^{\lambda }E=F^{\lambda +\lambda _{i,n}}E.
  \end{displaymath}
  Clearly if $\lambda \in A_{i,n}$ is a jump point of $(E,F)$, then
  $\lambda -\lambda _{i,n}$ it is a  jump point of $(E,F_{i,n})$. Since
  the $m_{i}$ are the equations of the facets of $C$, we are in the
  situation of Proposition
  \ref{prop:1}~\eqref{item:6}, thus the multifiltration $F_{i,n}$ is
  separated. By induction the number of jump points of $(E,F_{i,n})$
  is finite. Hence the number of jump points of $(E,F)$ is finite.
 \end{proof}

 The condition of being polyhedral is necessary in Proposition
 \ref{prop:2} as the following example shows.
 \begin{ex}
   Let $A=\Z^{2}$ and
   \begin{displaymath}
     C=\set{(m,n)\in A}{m\ge 0,\ 0\le n\le \sqrt{2}m}
   \end{displaymath}
   Let $E$ be a one dimensional vector space and 
   $F$ the multifiltration
   \begin{displaymath}
     F^{(m,n)}E=
     \begin{cases}
       E,&\text{ if }m<0, \sqrt{2}m \le n\le 0,\\
       0,&\text{ otherwise}.
     \end{cases}
   \end{displaymath}
   Then the multifiltration is separated but the number of jump points in
   infinite. Indeed, for each $m>0$, let $(p_{m},q_{m})$ the pair of
   positive integers with
   \begin{equation}
     \label{eq:5}
     q_{m}\le m, 0\le p_{m}\le \sqrt{2}q_{m}
   \end{equation}
   such that $p_{m}/q_{m}$ has strictly minimal distance to the line
   $y=\sqrt{2}x$ among the points satisfying conditions \eqref{eq:5}. Since
   this distance can be arbitrarily small but 
   never zero, the set of different $(p_{m},q_{m})$ is infinite. And,
   for each $(p_{m},q_{m})$ in that set,  the
   point $(-p_{m},-q_{m})$ is a jump point.

   The key of this example, besides the irrationality of the slope, is
   that $F^{(0,0)}E=0$. If 
   we change slightly the multifiltration writing $F^{(0,0)}E=E$, then
   the multifiltration has a single jump point.
 \end{ex}

 There is a converse to Proposition \ref{prop:2}.

 \begin{prop}
   Let $(A,C)$ be a preordered f.~g.~abelian group with $C$ generating, Let $(E,F)$
   be a multifiltered vector space with $F$ chain-separated and with a
   finite number of jump points. Then $F$ is separated.
   \label{prp:csfjissp}
 \end{prop}
 \begin{proof}
   Let $\setl{K^{\alpha }}$ be a collection of subsets with
   $\bigcap_{\alpha }\widetilde K_{\alpha }=\emptyset$. Since the
   set of jump points $\sJ$ is finite and $C$ is generating, there is
   an $a\in A$ such that $\lambda \le a$ for all $\lambda \in
   \sJ$. Then there is an $\alpha_{0} $ such that $a\not \in \widetilde
   K_{\alpha_{0} }$, which implies that $\sJ\cap \widetilde
   K_{\alpha_{0} }=\emptyset$. We claim that $F^{K_{\alpha
       _{0}}}E=\setl{0}$. If this is not the case, there is a $\lambda
   _{1}\in \widetilde K_{\alpha _{0}}$ such that $F^{\lambda
     _{1}}E\not = 0.$

   Since $\lambda _{1}$ is not a jump point, there is a $\lambda
   _{2}>\lambda _{1}$ with $F^{\lambda
     _{2}}E\not = 0 $. Repeating this process, we obtain an infinite
   sequence $\lambda _{1}<\lambda _{2}<\dots$ with $F^{\lambda
     _{i}}E\not = 0 $. By the chain-separatedness of $F$ we deduce
   that  $\bigcap_{i }F^{\lambda
     _{i}}E= 0 $. By the finite dimensionality of $E$, there is an
   $i_{0}$ such that $F^{\lambda
     _{i_{0}}}E= 0$ contradicting the construction of the $\lambda
   _{i}$. 
 \end{proof}

In Proposition \ref{prop:2}, the hypothesis that the cone is
polyhedral can be 
relaxed if we  add the condition of being regular. 

\begin{lem}
  Let $(A, C)$ be a preordered f.~g.~abelian group, with $C$ being a strict
  generating cone.  Suppose $(E, F)$ is a decreasing exhaustive
  regular separated multifiltration indexed
  by $(A, C)$.  Then the number of jump points is finite.
  \label{lem:finjumps}
\end{lem}
\begin{proof}
  Suppose $\sJ$ is the set of all jump points and it is infinite. As proved
  above, $E = \sum_{\lambda \in \sJ} F^{\lambda} E$. Since $E$ is finite
  dimensional, there exists a finite subset $\sJ_f \subset \sJ$ such that
  \begin{equation*}
    E = \sum_{\lambda \in \sJ_f} F^{\lambda} E.
  \end{equation*}
  We shall prove that $\sJ$ is bounded. The fact that $\sJ$ is bounded above
  follows from separatedness of $F$ (see Lemma \ref{lemm:3}).
  
  Suppose that $\sJ$ is not bounded below. Since $\sJ_{f}$ is finite
  and $C$ is a generating cone, there is an $a$ such that $a \leq \lambda$
  for all $\lambda \in \sJ_f$. Since $\sJ$ is not bounded below, we can
  choose a $\mu \in \sJ$ such that $a \nleq \mu$. By regularity and by
  the fact that $\mu$ is a jump point, we 
  see that
  \begin{equation*}
    F^{\mu} E \cap \left(
    \sum_{
	  \tau \in \sJ_{f}    }
    F^{\tau} E \right)
    =
    \sum_{
      \substack{
        \tau \in \sJ_{f}^{\sim} \\
        \tau \geq \mu
      }
    }
    F^{\tau} E
    \stackrel{*}{=}
    \sum_{
      \substack{
        \tau \in \sJ_{f}^{\sim} \\
        \tau > \mu
      }
    }
    F^{\tau} E
    \neq
    F^{\mu} E,
  \end{equation*}
  where the equality marked $*$ follows from the fact that $\mu \notin
  \sJ^{\sim}$.

  On the other hand,
  \begin{equation*}
    F^{\mu} E \cap \left( 
    \sum_{\tau \in \sJ_f} F^{\tau} E
    \right)
    = F^{\mu} E \cap E = F^{\mu} E,
  \end{equation*}
  which is a contradiction. Thus $\sJ$ is bounded below.

  In conclusion, $\sJ$ is bounded in $A$ and hence is finite.

%
%
\end{proof}

\begin{rmk}
  \begin{enumerate}
  \item If $A$ is a totally  ordered f.~g.~abelian group and $C$ is the cone of
	all positive elements, then every multifiltration by $(A, C)$ is regular.
  \item If $A$ is a lattice, in the sense of poset theory, (i.~e.~$C$
    is a unimodular cone) then 
    regularity implies that
    \begin{displaymath}
      E_{\lambda }\cap E_{\lambda '}=E_{\lambda \vee \lambda '}, 
    \end{displaymath}
    but, in general
    \begin{displaymath}
      E_{\lambda }+ E_{\lambda '}\not = E_{\lambda \land \lambda '}
    \end{displaymath}
  \end{enumerate}
\end{rmk}

The main result of the paper is the following result that
characterized when the associated graded vector space has the same
dimension as the original one.

\begin{thm}  \label{thm:cridsdec}
  Let $(A, C)$ be a preordered f.~g.~abelian group, where $C$ is a
  strict generating cone; and $(E,
  F)$ a multifiltration indexed by $(A, C)$, where $E$ is a finite
  dimensional vector space. The multifiltration $F$ is exhaustive, separated
  and regular if and only if there exists a (non-canonical) decomposition
  \begin{displaymath}
    E=\bigoplus _{\lambda \in A} E_{\lambda}
  \end{displaymath}
  such that
  \begin{displaymath}
	F^{\lambda} E = \bigoplus_{\mu \ge \lambda } E_{\mu}.
  \end{displaymath}
\end{thm}
\begin{proof}
  Assume that such a decomposition exists. Clearly the multifiltration is
  separated and exhaustive. Let $K_1,K_2\subset A$ be two subsets. Then
  \begin{align*}
    \sum_{\lambda \in K_1} F^{\lambda} E \cap \sum_{\mu \in K_2} F^{\mu} E
    &= \bigoplus_{\lambda' \in \widetilde{K}_{1}} E_{\lambda'} \cap
      \bigoplus_{\mu' \in \widetilde{K}_{2}} E_{\mu'} \\
	& = \bigoplus_{\mu' \in \widetilde{K}_{1} \cap \widetilde{K}_{2}}
	E_{\mu'},\\
    & = \sum_{\mu' \in \widetilde{K}_{1} \cap  \widetilde{K}_{2}} F^{\mu'} E.
  \end{align*}

  Conversely, assume that the multifiltration is exhaustive, separated and
  regular. Define
  \begin{displaymath}
	E_{\lambda}=F^{\lambda} E \left / \sum_{\mu > \lambda} F^{\mu} E \right
	.
  \end{displaymath}
  Write $p_{\lambda} \colon F^{\lambda} E \to E_{\lambda}$ the projection
  and choose a section $s_{\lambda} \colon E_{\lambda} \to F^{\lambda} E$.
  We identify $E_{\lambda}$ with $s_{\lambda }(E_{\lambda})\subset E$. But
  note that this identification is non-canonical.

  We have to prove:
  \begin{enumerate}
	\item \label{item:1} $E = \sum_{\lambda \in A} E_{\lambda}$;
	\item \label{item:2} $F^{\lambda} E = \sum_{\mu \ge \lambda } E_{\mu}$;
	\item \label{item:3} $\forall S', S'' \subset A$, with $S' \cap S'' =
	  \emptyset$,
	  \begin{displaymath}
		\left( \sum_{\lambda \in S'} E_{\lambda} \right ) \cap \left(
		\sum_{\mu  \in S''} E_{\mu} \right ) = \setl{0}.
  \end{displaymath}
\end{enumerate}
Since the multifiltration is assumed to be exhaustive, property
\eqref{item:2} implies property \eqref{item:1}.

For the sake of clarity we define the
following.  Let $S \subset A$. We say that $S$ satisfies condition $(*)$ if
\begin{displaymath}
  (*) \qquad \forall \lambda \in S, \ \mu  > \lambda ,\  \mu \text{ is a jump
    point}\qquad
  \Longrightarrow \qquad \mu \in S.
\end{displaymath}

We now prove \eqref{item:2}, so let $\lambda \in A$ and 
$x \in F^{\lambda} E$. Write $x_{0}=x$ and
\begin{displaymath}
  S_{0}=\left\{\mu \in A\,\middle | \, \mu \ge \lambda , \mu \text{ is a
  jump point} \right \}.
\end{displaymath}
Then $S_0$ is finite (from Lemma \ref{lem:finjumps}) and $x_{0} \in
\sum_{\mu \in S_{0}} F^{\mu} E$ (from lemmas \ref{lem:sepexhds} and
\ref{lem:jumpssum}). Note that, $S_0$ satisfies condition $(*)$.

Assume that we have a subset $S_{n}\subset S_{0}$ which satisfies condition
$(*)$, and an element $x_{n}\in \sum_{\mu \in S_{n}} F^{\mu} E$. Let
$\mu_{n,1},\dotsc,\mu_{n,k_{n}}$ the minimal elements of $S_{n}$.  This is a
nonempty finite list since $S_{n}$ is finite.  Write
\begin{displaymath}
  x_{n}=\sum_{\mu \in S_{n}} x_{\mu }, \quad x_{\mu } \in F^{\mu} E.
\end{displaymath}
For $i=1,\dots,k_{n}$, let
\[
  x_{n,i} = s_{\mu_{n,i}} p_{\mu_{n,i}} x_{\mu_{n,i}} \in E_{\mu_{n,i}}
  \text{ and }
  y_{n}=\sum_{i=1}^{k_{n}} x_{n,i}.
\]
Set $x_{n+1}=x_{n}-y_{n}$ and $S_{n+1}=S_{n} \setminus \{\mu _{n,1},\dots
,\mu_{n,k_{n}}\}$.  Then $S_{n+1} \subsetneq S_{n}$ and satisfies the
condition $(*)$.  Note
\begin{equation*}
  x_{n+1} = \sum_{i=1}^{k_n} \big( x_{\mu_{n,i}} - s_{\mu_{n,i}}
  p_{\mu_{n,i}} (x_{\mu_{n,i}}) \big) + \sum_{\lambda \in S_{n+1}}
  x_{\lambda}.
\end{equation*}
Since
\begin{displaymath}
	p_{\mu_{n,i}} \big(x_{\mu_{n,i}} - s_{\mu_{n,i}} p_{\mu_{n,i}}
	(x_{\mu_{n,i}}) \big) = 0,
\end{displaymath}
by lemmas \ref{lem:sepexhds}, \ref{lem:jumpssum} and the fact that
$S_n$ satisfies condition $(*)$,  we have
\begin{displaymath}
  x_{\mu_{n,i}} - s_{\mu_{n,i}} p_{\mu_{n,i}} (x_{\mu_{n,i}}) \in
  \sum_{\lambda \in S_{n+1}} F^{\lambda} E,
\end{displaymath}
and hence $x_{n+1} \in \sum_{\mu \in S_{n+1}} F^{\mu} E$. So we can iterate
the process. By the finiteness of $S_{0}$ there is an $n_{0}$ such that
$S_{n_{0}+1}=\emptyset$. Then
\begin{displaymath}
  x= \sum_{n=0}^{n_{0}}\sum_{i=1}^{k_{n}} x_{n,i}
\end{displaymath}
proving \eqref{item:2}.
  
  We next prove \eqref{item:3}. Assume that the intersection contains a
  point different from zero. Let
  \begin{displaymath}
	0 \neq x \in \left( \sum_{\lambda \in S'} E_{\lambda} \right )
    \cap
    \left( \sum_{\mu  \in S''} E_{\mu} \right ).
  \end{displaymath}
  Thus we can write
  \begin{displaymath}
	x = \sum_{\lambda \in S'} x_{\lambda} = \sum_{\mu \in S''} y_{\mu},
	\quad x_{\lambda } \in E_{\lambda}, \quad y_{\mu }\in E_{\mu}.
  \end{displaymath}
  By shrinking $S'$ and $S''$ if needed, we can assume that, $x_{\lambda
  }\not = 0 $ for all $\lambda \in S'$ and that $y_{\mu }\not = 0 $ for all
  $\mu \in S''$. Let $\lambda _{0}$ be a minimal element of $S'\cup S''$. To
  fix ideas assume that $\lambda _{0}\in S'$. For $\lambda \in S'\setminus
  \{\lambda _{0}\}$ write $y_{\lambda }= -x_{\lambda }$. Also write
  $S'''=(S'\cup S'')\setminus \{\lambda _{0}\}$. Then
  \begin{displaymath}
	x_{\lambda_{0}}=\sum_{\mu \in S'''} y_{\mu }.
  \end{displaymath}
  By the regularity of the multifiltration and the minimality of $\lambda_{0}$,
  we deduce
  \begin{displaymath}
	x_{\lambda _{0}}\in \sum_{\substack{\mu '\ge \lambda _{0}\\\mu' \in
	(S''')^{\sim}}} F^{\mu'} E \subset \sum_{\mu ' > \lambda _{0}} F^{\mu'}
	E.
  \end{displaymath}
  Since $x_{\lambda _{0}} \in E_{\lambda_{0}}$, $x_{\lambda_{0}}=0$
  reaching a contradiction. Therefore we deduce \eqref{item:3} and we
  complete the proof of the theorem. 
\end{proof}

\section{Equivariant sheaves on toric varieties}
\label{sec:equiv-sheav-toric}

Klyachko
\cite{Klyachko} has given a classification of equivariant vector
bundles on toric varieties. This 
classification has been extended to the classification of equivariant
quasi-coherent sheaves on toric varieties and that of equivariant
torsion-free sheaves by Perling \cite{Perling2004}.
The results of the previous section allow us to reformulate Perling's conditions for
an equivariant coherent sheaf to be 
locally free in terms of regularity.

\emph{Throughout this section, we shall be working over a fixed
  algebraically closed base field $k$.}

We fix a lattice $N \cong \Z^n$, $M = \Hom_{\Z}(N, \Z)$, $N_{\R} = N
\otimes_{\Z} \R$, $M_{\R} = M \otimes_{\Z} \R$, and a fan $\Sigma$ in
$\N_{\R}$.  For a cone $\sigma \in \Sigma$, $M(\sigma)$ will denote the
lattice $M / (\sigma^{\perp} \cap M)$. The projection $M \to M(\sigma)$ is
denoted by $\pi_{\sigma}$. For a cone $\sigma \in \Sigma$, $\dual{\sigma}$,
by abuse of notation, will denote $\dual{\sigma} \cap M$ and $\pi_{\sigma}
(\dual{\sigma})$ will stand for $\pi_{\sigma} (\dual{\sigma} \cap M)$.  Let
$X = X_{\Sigma}$ denote the toric variety defined by $\Sigma$. By $x_0 \in
X_{\Sigma}$ we denote the identity in the embedded torus $U_{\setl{0}} \subset
X_{\Sigma}$.

\begin{df}\label{def:11}
  An \emph{equivariant sheaf} $\caE$ on a toric variety $X_{\Sigma}$ is a
  sheaf of modules on $X_{\Sigma}$ along with a collection of isomorphisms
  \[
    \Phi_t \colon t^* \caE \xrightarrow{\cong} \caE
  \]
  for every $t \in \T = U_{\setl{0}}$, such that the following cocycle
  condition holds: for every pair $t, t' \in \T$, the diagram
  \[
    \xymatrix{
      t^{\ast} t'^{\ast} \caE \ar[dr]_{t^{\ast} \Phi_{t'}} &
      (t't)^{\ast} \caE \ar[l]_{\cong} \ar[r]^-{\Phi_{t't}} &
      \caE \\
      &
      t^{\ast} \caE \ar[ur]_{\Phi_t} &
    }
  \]
commutes.
\end{df}
\begin{rmk}
  If $\caE$ is an equivariant sheaf according to Definition \ref{def:11},
  and $U$ is an invariant open subset, 
  there is an induced action of $\T$
  on $\gs{U}{\caE}$ that, for a section $s \in \gs{U}{\caE}$, is given by 
  \[
    t \cdot s = \Phi_{t^{-1}}  ((t^{-1})^{\ast} s).
  \]
  This  action is similar to that used by
  Klyachko \cite{Klyachko}, Payne \cite{Payne:moduli} and others.
  However, this action differs from the one used in Perling, who
  considers a right
  action on $X$ and a left action on $\caE$.

  The action of $\T$ on $\gs{U}{\caE}$ induces a grading
  \begin{displaymath}
    \gs{U}{\caE}=\bigoplus_{m\in M}\grgs{U}{\caE}{m},
  \end{displaymath}
where $s\in \grgs{U}{\caE}{m}$ if and only if
\begin{displaymath}
  t\cdot s = \chi^{m}(t) s= t^{m}s.
\end{displaymath}
\end{rmk}

\begin{ex}\label{exm:6}
  The structural sheaf $\caO_{X}$ is an equivariant sheaf. If $U$ is
  an invariant set, the action of $\T$ of $\gs{U}{\caO_{X}}$ is given
  by
  \begin{displaymath}
    (t\cdot f)(x) = f(t^{-1}x).
  \end{displaymath}
  Therefore, the character $\chi^{m}\in \gs{\T}{\caO_{X}}$ belongs to
  $\grgs{\T}{\caO_{X}}{-m}$, as the following chain of equalities shows
  \begin{displaymath}
    (t\cdot \chi^{m})(x)=\chi^{m}(t^{-1}x)
    =t^{-m}x^{m}=(\chi^{-m}(t)\chi^{m})(x). 
  \end{displaymath}
\end{ex}

To describe equivariant quasi-coherent sheaves, Perling \cite{Perling2004}
used the concept of $\Delta$-families. We define an analogue of
$\Delta$-families for multifiltrations. Let $N$, $M$, $\Sigma$, and $M(\sigma)$
for $\sigma \in \Sigma$ be as in the beginning of this section. Note that, if
$\tau \face \sigma$, we have a natural quotient map $\pi_{\tau \sigma} :
M(\sigma) \to M(\tau)$. We note here that $\pi_{\sigma}(\dual{\sigma})$ is always
a strict cone, and hence induces a partial order on $M(\sigma)$.  To
ease the notation, 
the orders induced in the spaces $M(\sigma )$ will be denoted by
$\le$, without any decoration 
corresponding to the cone. The particular order one is
considering will be clear by the space to which the elements
belong. By contrast, for orders on $M$ we
shall add a decoration to $\le$ indicating the cone defining the order.

\begin{df}\label{def:10}
  Let $E$ be a vector space. A \emph{$\Sigma$-compatible family of
  multifiltrations} on $E$ is a collection of
  multifiltrations $\set{F_{\sigma}}{\sigma \in \Sigma}$ where
  $F_{\sigma}$ is indexed by $(M(\sigma),
  \pi_{\sigma}(\dual{\sigma}))$, and such that whenever $\tau \face \sigma$,
  \begin{displaymath}
    F_{\tau }=\Ind_{\pi_{\tau \sigma } }(F_{\sigma }).
  \end{displaymath}
  In other words,
  \[
	F_{\tau}^{\lambda} E = \sum_{
		\mu \colon
		\pi_{\tau\sigma} (\mu) \ge \lambda
	} F_{\sigma}^{\mu} E.
  \]
  
  A $\Sigma$-compatible family of multifiltrations is said to be
  \emph{exhaustive} 
  (respectively \emph{separated}, \emph{regular}) if $F_{\sigma}$ is
  exhaustive (respectively separated, exhaustive) for all $\sigma \in
  \Sigma$.
\end{df}

\begin{rmk}
  A $\Sigma $-compatible family of multifiltrations is completely
  determined by the multifiltrations $F_{\sigma }$ for top dimensional
  cones. The fact that, given a common face $\tau \prec \sigma $,
  $\tau \prec \sigma '$ it is necessary that
  \begin{displaymath}
    \Ind_{\pi _{\tau \sigma }}(F_{\sigma })=\Ind_{\pi _{\tau \sigma'
      }}(F_{\sigma' }) 
  \end{displaymath}
  determines the compatibility conditions among such multifiltrations.

  In view of propositions \ref{prop:4}, \ref{prop:1} and \ref{prop:3},
  in order to check that a $\Sigma $-compatible family of
  multifiltrations is exhaustive, separated or regular it is enough to
  check it on maximal cones.
\end{rmk}

Following Perling, we shall see that $\Sigma$-compatible families of
multifiltrations arise from $\T$-equivariant, coherent, torsion-free
sheaves.  
In order to compare with Perling's results, we adapt \cite[Definition
5.17]{Perling2004}  to decreasing filtrations. In order to distinguish
from Definition \ref{def:10} we call them $\Sigma $-compatible
systems. 

\begin{df}\label{def:12}
  Let $E$ be a vector space. A \emph{$\Sigma$-compatible system of
  multifiltrations} on $E$ is a collection of
  multifiltrations $\set{F_{\sigma}}{\sigma \in \Sigma}$ where
  $F_{\sigma}$ is indexed by $(M,\dual{\sigma})$, satisfying the
  conditions
    \begin{enumerate}
  \item \label{item:19} For every $\sigma \in \Sigma $, $E=\bigcup_{m\in  M}F_{\sigma
    }^{m}E$.
  \item \label{item:20} For each chain $\cdots <_{\sigma }
    m_{i-1}<_{\sigma } m_{i}<_{\sigma }\cdots$ of characters in $M$
    there is an $i_{0}$ such that $F^{m_{i}}E=0$ for all $i\ge i_{0}$.
  \item \label{item:21} There are only finitely many vector spaces
    $F^{m}_{\sigma }E$ that are not contained in the sum of $F_{\sigma
    }^{m'}$ with $m<_{\sigma }m'$.
  \item \label{item:22} (Compatibility condition) For each $\tau \prec
    \sigma $, with $\tau ^{\vee}=\sigma ^{\vee}+\Z_{\ge 0}(-m_{\tau
    })$ and $m_{\tau }\in \sigma ^{\vee}$, consider the 
    chains  $m-i\cdot m_{\tau }$ for $i\ge 0$. It is decreasing with
    respect to the order $\le_{\sigma }$. By condition \eqref{item:21}
    and the fact that $E$ is finite dimensional, this sequence becomes
    stationary for some $i_{\tau }^{m}\in \Z$. The compatibility
    condition requires that $F^{m}_{\tau }E=F^{m-i_{\tau
      }^{m}\cdot m_{\tau }}_{\sigma }$ for all $m\in M$.
  \end{enumerate}
\end{df}

\begin{lem} \label{lemm:8}
  Let 
  $\set{F_{\sigma}}{\sigma \in \Sigma}$ be a $\Sigma $-compatible system
  of multifiltrations on a finite dimensional vector space $E$,
  then $\set{\Ind_{\pi _{\sigma}}(F_{\sigma})}{\sigma \in \Sigma}$ is
  an exhaustive and separated $\Sigma $-compatible family of
  multifiltrations. Conversely, if 
  $\set{F_{\sigma}}{\sigma \in \Sigma}$ is a $\Sigma $-compatible
  family, that is separated and exhaustive, then $\set{\Res^{\pi _{\sigma
      }}(F_{\sigma})}{\sigma \in \Sigma}$ is a $\Sigma $-compatible
  system.  
\end{lem}
\begin{proof}
  Let $\set{F_{\sigma}}{\sigma \in \Sigma}$ a $\Sigma $-compatible system
  of multifiltrations on $E$. Condition \ref{item:19} implies that,
  for all $\sigma $, $F_{\sigma }$ is exhaustive. By Proposition
  \ref{prop:4}~\eqref{item:11} the multifiltrations $\Ind_{\pi_{\sigma
    }}F_{\sigma }$, $\sigma \in \Sigma $ are exhaustive. Condition
  \ref{item:20} implies 
  that, for each $\sigma $ the multifiltration $F_{\sigma }$ is
  chain-separated.  By Proposition \ref{prop:5} the multifiltration
  $\Ind_{\pi_{\sigma }}F_{\sigma }$ is chain-separated. Condition
  \ref{item:21} implies that  the multifiltrations $\Ind_{\pi_{\sigma
    }}F_{\sigma }$, $\sigma \in \Sigma $ have a finite number of jump
  points. By Proposition \ref{prp:csfjissp} the multifiltrations $\Ind_{\pi_{\sigma
    }}F_{\sigma }$ are separated.

  Next we have to show that condition \ref{item:22} implies that, for
  $\tau \prec \sigma $ the equality
  \begin{displaymath}
    \Ind_{\pi _{\tau \sigma} }(\Ind_{\pi _{\sigma }}(F_{\sigma })) =
    \Ind_{\pi _{\tau  }}F_{\tau  }   
  \end{displaymath}
  holds. Let $m '\in M(\tau )$ and choose $m\in M$ with
  $\pi _{\tau }(m )=m'$. Then
  \begin{displaymath}
    \Ind_{\pi _{\tau \sigma} }(\Ind_{\pi _{\sigma }}(F_{\sigma }))^{m'}E
    = \sum_{\mu \colon m'\le \pi _{\tau \sigma }(m)}
    \Ind_{\pi _{\sigma }}(F)^{\mu }E
    = \sum _{\mu \colon m \underset{\dual{\tau }}{\le}\mu }F^{\mu }_{\sigma }E
  \end{displaymath}
  Note that, in the above expression, $m'\le \pi _{\tau \sigma
  }(\mu )$ relates two elements of $M(\tau )$, therefore the only
  order that makes sense is the one given by the cone $\pi _{\tau
  }(\dual{\tau })$ and is not necessary to make it explicit. By
  contrast, the expression  $m \underset{\dual{\tau }}{\le}\mu$
  involves two elements of $M$ where we consider different
  orders. Therefore the need to make the order explicit.

  By condition \ref{item:22},
  \begin{displaymath}
    \Ind_{\pi_{\tau }}(F_{\tau })^{m '}E=F_{\tau }^{m} E=F_{\sigma
    }^{m-i_{\tau }^{m}\cdot m_{\tau }}E,
  \end{displaymath}
  where $m_{\tau }\in \dual{\sigma }$ satisfying $\dual{\tau
  }=\dual{\sigma }+\Z_{\le 0} m_{\tau }$ and $i_{\tau }^{m}$ satisfies
  that, for all $k\in \Z$,
  \begin{displaymath}
    F_{\sigma }^{m-k\cdot m_{\tau }}E \subset
      F_{\sigma
    }^{m-i_{\tau }^{m}\cdot m_{\tau }}E.
  \end{displaymath}
  Note that $m_{\tau }\in \orth{\tau}$, hence
  $m\underset{\dual{\tau}}{\le} m-i^{m}_{\tau }\cdot m_{\tau }$. Which
  implies
  \begin{displaymath}
     \Ind_{\pi_{\tau }}(F_{\tau })^{m '}E\subset \Ind_{\pi _{\tau
         \sigma} }(\Ind_{\pi _{\sigma }}(F_{\sigma }))^{m'}E.
  \end{displaymath}
  For the reverse inclusion, note that, if $m \underset{\dual{\tau
    }}{\le}\mu$ then $\mu -m\in \dual{\tau }$. By the properties of
  $m_{\tau }$, there is a $k\ge 0$ such that $\mu -m+k\cdot m_{\tau
  }\in \dual{\sigma }$, so $m-k\cdot m_{\tau } \underset{\dual{\sigma 
    }}{\le}\mu$. Therefore
  \begin{displaymath}
    F^{\mu }_{\sigma }E\subset F^{m-k\cdot m_{\tau }}_{\sigma
    }E\subset F_{\sigma
    }^{m-i_{\tau }^{m}\cdot m_{\tau }}E,
  \end{displaymath}
  from which we deduce 
  \begin{displaymath}
     \Ind_{\pi_{\tau }}(F_{\tau })^{m '}E\supset \Ind_{\pi _{\tau
         \sigma} }(\Ind_{\pi _{\sigma }}(F_{\sigma }))^{m'}E.
  \end{displaymath}
  In conclusion $\set{\Ind_{\pi _{\sigma }}(F_{\sigma })}{\sigma \in
    \Sigma }$ is a $\Sigma $-compatible family of multifiltrations on
  $E$ that is separable and exhaustive.

  Conversely, assume that $\set{F_{\sigma }}{\sigma \in
    \Sigma }$ is a exhaustive and separable $\Sigma $-compatible
  family of multifiltrations on 
  $E$. By Proposition \ref{prop:4}~\ref{item:12} the multifiltrations
  $\Res ^{\pi _{\sigma }}F_{\sigma }$ are exhaustive, hence they
  satisfy condition \ref{item:19}. By Lemma
  \ref{lemm:2} and Proposition  \ref{prop:5}  the multifiltrations
  $\Res ^{\pi _{\sigma }}F_{\sigma }$ are
  chain-separated, so they satisfy condition \ref{item:20}. By
  Proposition  \ref{prop:2} the multifiltrations $F_{\sigma }$ have a
  finite number of jump points, which implies that the multifiltrations  
  $\Res ^{\pi _{\sigma }}F_{\sigma }$  satisfy condition
  \ref{item:21}.  
  
  We need to prove that the multifiltrations $\Res ^{\pi _{\sigma
    }}F_{\sigma }$ satisfy condition \ref{item:22}. This amounts to
  prove that, for all $m\in M$,
  \begin{equation} \label{eq:6}
    \Res^{\pi _{\tau }}(F_{\tau })^{m}E=
     \sum_{k\in \Z_{\ge 0}}\Res^{\pi _{\sigma }}(F_{\sigma })^{m-k\cdot m_{\tau }}E.
   \end{equation}
   On the one hand,
   \begin{displaymath}
     \Res^{\pi _{\tau }}(F_{\tau })^{m}E
     = F_{\tau }^{\pi _{\tau }(m)}E
     =\sum_{\mu\colon \pi _{\tau }(m)\le \pi _{\tau \sigma }(\mu )}F^{\mu }_{\sigma }E,
   \end{displaymath}
   where the last equality follows form the $\Sigma $-compatibility of
   the family. On the other hand
   \begin{displaymath}
     \sum_{k\in \Z}\Res^{\pi _{\sigma }}(F_{\sigma })^{m+k\cdot
       m_{\tau }}E
     =
     \sum_{k\in \Z_{\ge 0}}F_{\sigma }^{\pi _{\sigma }(m-k\cdot m_{\tau })}E.
   \end{displaymath}
   Since $\pi _{\tau \sigma }(\pi _{\sigma }(m-k\cdot m_{\tau }))=\pi
   _{\tau }(m)$  we deduce the inclusion $\supset$ is equation
   \eqref{eq:6}. As discussed before, for each $\mu \in M(\sigma )$
   with $\pi _{\tau  }(m)\le \pi _{\tau \sigma }(\mu )$, there is a
   $k\ge 0$ such that $\pi _{\sigma }(m-k\cdot m_{\tau })\le \mu $,
   and the inclusion $\subset $ in \eqref{eq:6} follows. Therefore
   $\set{\Res^{\pi _{\sigma }}F_{\sigma }}{\sigma \in \Sigma }$ is a
   $\Sigma $-compatible system of multifiltrations. 
 \end{proof}

To describe reflexive sheaves we will need the following definition.
\begin{df}
  Let $\Sigma$ be a fan and $F_{\Sigma} = \set{F_{\sigma}}{\sigma \in
  \Sigma}$ be a $\Sigma$-compatible family of multifiltrations on a
finite dimensional vector space $E$.
  $F_{\Sigma}$ is said to be \emph{reflexive} if for all $\mu \in M$ and for
  all $\sigma \in \Sigma$,
  \[
	F_{\sigma}^{\mu} E = \bigcap_{\rho \in \sigma(1)}
	F_{\rho}^{\pi _{\rho \sigma}(\mu)} E,
  \]
  where $\sigma(1)$ is the collection of rays in $\Sigma(1)$ which are faces
  of $\sigma$.
\end{df}

The following result is a reformulation of Perling and Klyachko
results in the language of multifiltrations. 

\begin{thm} \label{thm:1}
  Suppose $X_{\Sigma}$ is a toric variety over $k$ given by a fan
  $\Sigma$ on a lattice $N$. Then the following holds:
  \begin{enumerate}
	\item \label{itm:torsfree}
	  The category of equivariant, coherent, torsion free sheaves on
	  $X_{\Sigma}$ is equivalent to the category of finite
          dimensional vector spaces provided with a separated, exhaustive
	  $\Sigma$-compatible family of  multifiltrations.
	\item \label{itm:reflexive}
	  The category of equivariant, coherent, reflexive sheaves on
	  $X_{\Sigma}$ is equivalent to the category of  finite
          dimensional vector spaces provided with a separated, exhaustive,
	  reflexive $\Sigma$-compatible family of multifiltrations.
	\item \label{itm:locfrees}
	  The category of equivariant, coherent, locally free sheaves on
	  $X_{\Sigma}$ is equivalent to the category of finite
          dimensional vector spaces provided with a  separated, regular,
	  exhaustive $\Sigma$-compatible family of multifiltrations.
  \end{enumerate}
\end{thm}
\begin{proof}
  Let $\caE$ be an equivariant torsion free coherent sheaf on
  $X_{\Sigma }$. 
We start by recalling the construction that associates to $\caE$ a
vector space with a collection of 
multifiltrations.
Let $E=\caE_{x_{0}}$ be the fibre at zero of $\caE$. Since $\caE$ is
coherent, this is a finite dimensional $k$-vector space.

Let $U$ be a $\T$-invariant open set in $X$. The action of $\T$ on
$\gs{U}{\caE}$ induces an $M$ grading
\[
  \gs{U}{\caE} = \bigoplus_{m \in M} \gs{U}{\caE}_{m},
\]
where $\gs{U}{\caE}_{m}$ is the eigenspace corresponding to the
character $\chi^{m}$.

Thus for each $\sigma \in \Sigma$, let $U_{\sigma}$ be the affine open
subset
\begin{displaymath}
  U_{\sigma} = \Spec k [\dual{\sigma} \cap M] \hookrightarrow X.
\end{displaymath}
By Perling
\cite[page 190]{Perling2004}, $\grgs{U_{\sigma}}{\caE}{m}$ injects into
$E$ by sending a section $s$ to its value $s(x_{0})\in \caE_{x_{0}}=E$. Denote the
image of this injection by $F_{\sigma}^{m}E$. This defines a decreasing 
multifiltration on $E$ indexed by $(M, \dual{\sigma})$. To see this we
have to show that, for $m\in M$ and $\mu \in \dual{\sigma }$, the
inclusion $F^{m}_{\sigma }E\subset F^{m-\mu }_{\sigma } E$
holds. Indeed, if $v\in F_{\sigma }^{m}E$, there is a section $s\in
\grgs{U_{\sigma  }}{\caE}{m}$ with $s(x_{0})=v$. According to Example
\ref{exm:6}, the character $\chi^{\mu }$ belongs to  $\grgs{U_{\sigma }}{\caO}{-\mu
}$.   Therefore $\chi^{\mu}s\in \grgs{U_{\sigma  }}{\caE}{m-\mu
}$. Since
\begin{displaymath}
  (\chi^{\mu}s)(x_{0})=\chi^{\mu}(x_{0})s(x_{0})=1v=v,
\end{displaymath}
we deduce that $v\in F_{\sigma }^{m-\mu }E$.

By \cite[Theorem 5.18]{Perling2004}, the assignment $\caE\mapsto
\set{F_{\sigma}}{\sigma \in \Sigma }$ is an equivalence of categories
between the category of equivariant torsion free coherent sheaves on
$X_{\Sigma }$ with the category of finite dimensional vector bundles
over $k$ provided with a $\Sigma $-compatible system of
multifiltrations.  By Lemma \ref{lemm:8}, this last category is
equivalent to the category of finite dimensional vector spaces
provided with an exhaustive and separated $\Sigma $-compatible family
of multifiltrations. This completes the proof of \eqref{itm:torsfree}.

  Statement \eqref{itm:reflexive} follows directly from \cite[theorem
  5.19]{Perling2004}. Thus it remains to prove \eqref{itm:locfrees}.

  Perling, in the proof of \cite[Theorem~5.22]{Perling2004}, where he deduces
  Klyachko's classification result, shows that an equivariant coherent
  reflexive sheaf
  corresponding to a $\Sigma $-compatible family of multifiltrations
  $\setl{F_{\sigma }}$ 
  is locally free if and only if, for each $\sigma \in \Sigma$, there
  exists a decomposition 
  \[
	E = \bigoplus_{m \in M(\sigma )} E_m^{\sigma} \quad \text{ such that
        }\quad  F_{\rho }^{i}E =
	\sum_{m \colon \pi _{\rho \sigma }(m) \geq i} E_m^{\sigma}
  \]
  for each ray $\rho \in \sigma(1)$.

  Assume that for $\sigma \in \Sigma $ such decomposition
  exists. Using that the sheaf is reflexive,
  \begin{displaymath}
    F^{m}_{\sigma }E
    =\bigcap_{\rho \in \sigma (1)}F_{\rho }^{\pi _{\rho \sigma }(m)}E
    =\bigcap_{\rho \in \sigma (1)}\ 
    \sum_{\mu \colon \pi _{\rho \sigma}(m)\le \pi _{\rho \sigma}(\mu )}E_{\mu }^{\sigma }.
  \end{displaymath}
  Given $m,\mu \in M(\sigma )$,  the following two condition
  \begin{description}
  \item[A] \label{item:A} $\forall \rho \in \sigma (1), \ \pi _{\rho 
    \sigma}(m)\le \pi _{\rho \sigma}(\mu )$
\item[B]  \label{item:B} $m\le \mu $
\end{description}
are equivalent. Therefore
   \begin{displaymath}
    F^{m}_{\sigma }E= 
    \sum_{\mu \colon m\le \mu }E_{\mu }^{\sigma }.
  \end{displaymath}
  By Theorem \ref{thm:cridsdec}, since the family of multifiltrations is
  exhaustive and reflexive, this implies that the multifiltration
  $F_{\sigma }$ is regular.

  Conversely, assume that  $\caE$ is an equivariant coherent torsion
  free sheaf  such that the corresponding exhaustive and separated
  $\Sigma $-compatible family of multifiltrations is regular. By
  Theorem  \ref{thm:cridsdec}, for each $\sigma \in \Sigma $ there is
  a decomposition $E=\bigoplus _{m\in M(\sigma )}E^{\sigma }_{m}$ such
  that
  \begin{displaymath}
    F_{\sigma }^{m}E=\bigoplus _{\mu \colon m\le \mu }E^{\sigma }_{\mu}.
  \end{displaymath}
  If $\rho \in \sigma (1)$ is a ray then
  \begin{equation}\label{eq:7}
    F_{\rho }^{i}E=\sum _{\mu\colon i\le \pi_{\rho\sigma }(\mu
      )}F_{\sigma }^{\mu }E
    =\sum _{\mu\colon i\le \pi_{\rho\sigma }(\mu
      )}E^{\sigma }_{\mu }.
  \end{equation}
  Using again the equivalence of conditions \textbf{A} and
  \textbf{B}, we deduce that the
  family of multifiltrations is reflexive, hence the sheaf $\caE$ is
  reflexive and, by \eqref{eq:7} is locally free. 
\end{proof}

\begin{rmk}\label{rem:3}
  It is instructive to write down Theorem \ref{thm:1} in the case of
  affine toric varieties. In this case we do not need to worry about
  compatibility conditions among the multifiltrations because all of
  them are induced by a single multifiltration corresponding to the
  maximal cone.

  Let $(M,C)$ be a torsion free finitely
  generated partially ordered abelian group and $C$ a generating polyhedral
  cone. Write $N=M^{\vee}$, $\sigma =C^{\vee}$, $X_{\sigma }$ the
  corresponding affine toric variety and $x_{0}\in X_{\Sigma }$ the
  distinguished point. Let $E$ be a $k$
  vector space of rank $r$. Then
  \begin{enumerate}
  \item the set of separated exhaustive $(M,C)$-multifiltrations on $E$ is
    in bijection with the isomorphisms classes of torsion free sheaves $\caE$
    on $X_{\sigma }$ with an isomorphism $\caE_{x_{0}}\cong E$; 
  \item the set of reflexive separated exhaustive $(M,C)$-multifiltrations on $E$ is
    in bijection with the isomorphisms classes of torsion sheaves $\caE$
    on $X_{\sigma }$ with an isomorphism $\caE_{x_{0}}\cong E$;
  \item the set of regular separated exhaustive $(M,C)$-multifiltrations on $E$ is
    in bijection with the isomorphisms classes of locally free sheaves $\caE$
    on $X_{\sigma }$ with an isomorphism $\caE_{x_{0}}\cong E$.
  \end{enumerate}
\end{rmk}

We end this section showing how Example \ref{exm:4} give us examples
of reflexive sheaves that are not locally free.
\begin{ex}\label{exm:3}
  First we give an example of a reflexive sheaf on a singular surface
  that has generic rank one and is not locally free.
  Let $N=\Z^{2}$ and consider the cone $\sigma
  =\cone((1,0),(-1,2))$. Then $\dual{\sigma } \cap M$ is the semigroup of
  Example \ref{exm:4}~\eqref{item:17}. Let $\Sigma $ be the non
  complete fan given by all the faces of $\sigma $. Then $X_{\Sigma }$
  is a singular affine toric variety and $U_{\sigma }=X_{\Sigma
  }$. Let $E$ be a two 
  dimensional vector space with basis $\setl{u,v}$ and let $F_{\sigma
  }$ be the filtration of Example \ref{exm:4}~\eqref{item:17}. For
  $\tau \prec \sigma $,
  write $F_{\tau }=\Ind_{\pi _{\tau \sigma }}(F_{\sigma })$. Then 
  $\set{F_{\tau }}{\tau \prec \sigma }$ is an exhaustive separated
  regular $\Sigma $-compatible family of multifiltrations. The
  corresponding sheaf $\caE$ is locally free. Let $S\subset E$ be the
  subspace generated by $\{u+v\}$. The
  multifiltration induced in $S$ is not regular. Let $\caS$ be the
  corresponding torsion free sheaf. we know that it is not locally
  free but is easy to check that is reflexive using for instance
  Theorem \ref{thm:1}~\eqref{itm:reflexive}. 
\end{ex}

\begin{ex} \label{exm:5}
  We next give an example of a reflexive, but non locally free, sheaf on a smooth
  threefold, that has generic rank 2.

  Let $N=\Z^{3}$ be a rank $3$ lattice and $\sigma $ the positive
  octant. To $\sigma $ corresponds the affine smooth toric variety
  $X_{\sigma }=\A^{3}$. Let $\Sigma $ be the corresponding fan. 
  Let $S$ be the subspace of Example \ref{exm:4}~\eqref{item:18} and
  $F$ the induced $\Sigma $-compatible family of multifiltrations as in
  Example \ref{exm:3}. This family is easily seen to be exhaustive,
  separated and reflexive, but is not regular. Thus it defines a
  reflexive sheaf that is not locally free. 
  \end{ex}

\section{Pull back of coherent torsion free equivariant sheaves}

In this section we will use the classification of torsion free sheaves in
terms of families of multifiltrations to describe the inverse image of
such sheaves by an equivariant morphisms

Let $N$ and $\dn{N}$ be two lattices and $\Sigma$ and $\dn{\Sigma}$ be
fans in $N_{\R}$ and $\dn{N}_{\R}$ respectively. Let $\varphi
\colon N \to \dn{N}$ be a map of lattices satisfying that, for each
$\sigma \in \Sigma $ there is a $\dn{\sigma }\in \dn{\Sigma }$ with
$\varphi(\sigma )\subset \dn{\sigma }$.  Let $M$ and
$\dn{M}$ be the dual lattices and let $\dual{\varphi}\colon M'\to M$
be the dual map.

The lattice map $\varphi$ induces 
a toric
morphism $\bar{\varphi} \colon X = X_{\Sigma} \to \dn{X} = X_{\dn{\Sigma}}$. Let
$\T$ (respectively, $\dn{\T}$) be the full dimensional torus in $X$
(respectively, $\dn{X}$). Suppose $\dn{\caE}$ is an equivariant
coherent torsion free sheaf 
on $\dn{X}$ corresponding  to the exhaustive separated $\dn{\Sigma
}$-compatible family of multifiltrations $\set{(\dn{E},
  \dn{F}_{\dn{\sigma}})}{\dn{\sigma} \in \dn{\Sigma}}$.

Let $\caE
= \bar{\varphi}^* \dn{\caE}$ be the pull back sheaf on $X$. It is again an
equivariant coherent torsion free sheaf. Let
$\set{(E, F_{\sigma})}{\sigma \in \Sigma}$ be the corresponding $\Sigma $-compatible
family of filtrations.

\begin{prop}   \label{prp:pbforaff}
  In the previous set up, let
  $\sigma \in \Sigma $ and $\dn{\sigma} \in \dn{\Sigma} $ be
  cones such that $\bar{\varphi}(\sigma) \subset
  \dn{\sigma}$. Denote by $\dual{\varphi}_{\sigma }\colon M'(\sigma'
  )\to M(\sigma )$ the map induced by $\dual{\varphi}$.

  Then
  \begin{displaymath}
    F_{\sigma }= \Ind_{\dual{\varphi}_{\sigma }}(F_{\sigma '})
  \end{displaymath}
\end{prop}
\begin{proof}
  Let $U_{\sigma}\coloneqq \Spec(k[\dual{\sigma }\cap M]$ and
  $U_{\dn{\sigma}}\coloneqq \Spec(k[\dual{\dn{\sigma}{}}\cap \dn{M}]$
  be the invariant affine subsets corresponding to the cones $\sigma $
  and $\sigma '$. Then $\caE|_{U_{\sigma }} =
  \widetilde{H}$ (respectively, $\dn{\caE}|_{U_{\dn{\sigma} }} =
  (\dn{H})^{\sim}$) where $H= \gs{U_{\sigma }}{\caE}$ 
  is an $A_{\sigma} \coloneqq k[\dual{\sigma} \cap M]$-module (respectively,
  $\dn{H}=\gs{U_{\dn{\sigma} }}{\dn{\caE}}$ is an $A_{\dn{\sigma}} :=
  k[\dual{{\dn{\sigma}}} \cap 
  \dn{M}]$-module).  By the construction of the pull-back sheaf, $H =
  \dn{H} \otimes_{A_{\dn{\sigma}}} 
  A_{\sigma}$. Furthermore, $E = H \otimes_{A_{\sigma}}
  k(x_{0})=\dn{H} \otimes_{A_{\dn{\sigma}}}  k(\dn{x}_{0})$.

  Note that $k(x_{0})$ is the field $k$ with the structure of
  $A_{\sigma }$-module given by the morphism $A_{\sigma }\to k$,
  $\chi^{m}\mapsto 1$, which is the evaluation at the point $x_{0}$. 

  By the construction of the family of multifiltrations from the
  torsion free sheaf (see the proof of Theorem \ref{thm:1}) we have
  \[
    F^{m}_{\sigma}E = \im(H_{m} \hookrightarrow H \to H
    \otimes_{A_{\sigma}} k(x_0)),
  \]
  where $H_{m}$ is the eigenspace of character $\chi^{m}$ of $H$.
  But note that,
  \begin{displaymath}
    H_{m} = \left( \dn{H} \otimes_{A_{\dn{\sigma}}}
    A_{\sigma}\right)_{m} = \bigoplus_{\dn{\mu} \in \dn{M}} \left(
    \dn{H} \right)_{\dn{\mu}} \otimes \left( A_{\sigma} \right)_{m
      - \dual{\varphi}(\dn{\mu})} 
  \end{displaymath}
  Thus,
  \begin{align*}
    F_{\sigma}^{m}E 
    &=
    \im(H_{m} \to H \otimes_{A_{\sigma}} k(x_0)) \\
    &=
    \im\left( \bigoplus_{\dn{\mu} \in \dn{M}} \left( \dn{H}
    \right)_{\dn{\mu}} \otimes_k \left( A_{\sigma} \right)_{m -
      \dual{\varphi}(\dn{\mu})} \to \dn{H} \otimes_{A_{\dn{\sigma}}} k(x_0)
      \right) \\
  \end{align*}
  Recall that, by Example \ref{exm:6}, the isotypical component
  $(A_{\sigma})_{m}$ is non zero if and only if $-m\in \dual{\sigma
  }$. Therefore, in the direct sum above, the only non zero terms 
  correspond to those $\mu$ for which $\dual{\varphi}(\dn{\mu}) - m \in
  \dual{\sigma}$, or in other words, for which $\dn{\mu}
  \geq_{\dual{\sigma}} m$. Thus,
  \begin{align*}
    F_{\sigma}^{\lambda}E 
    &=
    \im \left( \bigoplus_{\dn{\mu} \colon  \dual{\varphi}(\dn{\mu})
      \underset {\dual{\sigma}}{\geq}
    m} \dn{H}_{\dn{\mu}} \otimes_k (A_{\sigma})_{m -
    \dual{\varphi}(\dn{\mu})} \to \dn{H} \otimes_{A_{\dn{\sigma}}} k(x_0)
    \right) \\
    &=
    \sum_{\dn{\mu} \colon \dual{\varphi}(\dn{\mu}) \underset{\dual{\sigma}}{\geq} m}
      F{\dn{\sigma}}^{\dn{\mu}}E\\
    &=\Ind_{\dual{\varphi}}(F_{\sigma '})^{m} E,
  \end{align*}
  since the image of $(A_{\sigma})_{\lambda - \fbd(\dn{\mu})}$ in
  $k(x_0)$ is non-zero. This completes the proof.
\end{proof}

In order to patch together the different results of Proposition \ref{prp:pbforaff} for
different $\sigma $ we make the following definition.

\begin{df}With the setting at the beginning of this
  section, we denote by $\Ind_{\dual{\varphi}}\set{F_{\dn {\sigma
      }}}{\dn{\sigma }\in \dn{\Sigma }}$ the $\Sigma $-compatible
  family, that to each $\sigma \in \Sigma $ assigns the
  multifiltration $\Ind_{\dual{\varphi}_{\sigma }}(F_{\dn{\sigma }})$
  for any $\dn{\sigma }\in \dn{\Sigma }$ with $\varphi(\sigma )\subset
  \dn{\sigma }$.

  By Proposition \ref{prop:7}, it is clear that
  $\Ind_{\dual{\varphi}_{\sigma }}(F_{\dn{\sigma }})$ 
  does not depend on the choice of $\dn{\sigma }$ with $\varphi(\sigma )\subset
  \dn{\sigma }$ and that we obtain a $\Sigma $-compatible family.  
\end{df}

\begin{thm}
  Let $\varphi \colon N \to \dn{N}$ be a morphism of lattices compatible
  with fans $\Sigma$ on $N_{\R}$ and $\dn{\Sigma}$ on $\dn{N}_{\R}$.
  Denote the dual map by $\dual{\varphi} \colon \dn{M} \to M$. Let
  $\bar{\varphi} \colon X = 
  X_{\Sigma} \to \dn{X} = X_{\dn{\Sigma}}$ be the induced toric morphism
  between the corresponding toric varieties. Suppose $\dn{\caE}$ is an
  equivariant coherent torsion free sheaf on $\dn{X}$  and $\caE =
  \bar{\varphi}^{\ast} 
  \dn{\caE}$ the pull-back sheaf on $X$. If $\dn{\caE}$ corresponds
  to the $\dn{\Sigma }$-compatible family of multifiltrations
  $\set{\dn{F}_{\dn{\sigma}}}{\dn{\sigma} \in \dn{\Sigma}}$ on
  $E\coloneqq \dn{\caE}_{x_{0}}$, then $\caE$ corresponds to the
  $\Sigma $-compatible family of multifiltrations
  $\Ind_{\dual{\varphi}}\set{\dn{F}_{\dn{\sigma}}}{\dn{\sigma} \in
    \dn{\Sigma}}$. 
\end{thm}

\section{Representations of solvable groups}
In this section, we shall see another situation in which multifiltrations
arise. Let $G$ be a connected solvable algebraic group over an algebraically
closed field $k$. We know, from \cite[chapter 19]{Humphreys} that such a
solvable group fits into a split short exact sequence
\begin{equation} \label{eqn:sessolgp}
  0 \rightarrow U \xrightarrow{\iota} G \xrightarrow{\pi} \T \rightarrow 0,
\end{equation}
with $U$ being the unipotent radical of $G$ and $\T$ a torus.  Choose a
splitting $\alpha \colon \T \to G$. Then, via $\alpha$, $\T$ acts on $U$ by
conjugation: $t \cdot u = \alpha(t) u \alpha(t)^{-1}$ and this induces the
adjoint representation $\Ad$ of $\T$ on the Lie algebra $\fu$ of
$U$.

Let $M = \Hom(\T, \Gm)$. The action of $\T$ on $\fu$ induces a decomposition
\[
  \fu = \bigoplus_{m \in M} \fu_{m}^{\alpha}.
\]
Let $R = \set{m \in M}{\fu_{m}^{\alpha} \neq \setl{0}}$, which is a finite set.
Consider the cone $C' \subset M_{\Q}$ defined by
\[
  C' = \set{m \in M_{\Q}}{m = \sum_{r \in R} a_r r,\text{ where } 0 \leq a_r
  \in \Q}.
\]
and let $C=C'\cap M$.

\begin{lem}\label{lem:indepcon}
  In the aforementioned set up, $C$ is independent of the choice of the
  splitting $\alpha$.
\end{lem}

\begin{proof}
  Suppose $\alpha$ and $\alpha'$ are two sections of $\pi$.  Recall that any
  two tori in $G$ are conjugate.  Hence, there exists a $g \in G$ such that
  $\alpha'(\T) = g \alpha(\T) g^{-1}$. Let $\varphi_g : G \to G$ be the
  conjugation map $x \mapsto g x g^{-1}$. By definition, $\Ad(g) = (d
  \varphi_g)_{e}$, where $e \in G$ is the identity.

  Let us denote the decomposition of $\fu$ obtained from $\alpha(\T)$ and
  $\alpha'(\T)$ as
  \[
	\fu = \bigoplus_{m \in R_{\alpha} \subset M} \fu_{m}^{\alpha} \qquad
	\qquad \fu = \bigoplus_{m' \in R_{\alpha'} \subset M}
	\fu_{m'}^{\alpha'}.
  \]
  We claim that $R_{\alpha'} = R_{\alpha}$. But that is obvious as if $u \in
  \fu_{m}^{\alpha'}$, $\Ad(g) u \in \fu_{m}^{\alpha}$\remvivek{
	Let $u \in \fu_{m}^{\alpha'}$. Then $\Ad(\alpha'(t))u =
	m(\pi(\alpha'(t))) u = m(t) u$. Now $\Ad(\alpha(t)) (\Ad(g) u) =
	d\varphi_{\alpha(t)} d\varphi_g u = d \varphi_{\alpha(t)g} u = d
	\varphi_{g \varphi_{g^{-1}}(\alpha(t))} u = d \varphi_g d
	\varphi_{\alpha'(t)} u = \Ad(g) \Ad(\alpha'(t)) u = \Ad(g) (m(t) u) =
	m(t) \Ad(g) u$ proving that $\Ad(g) u \in \fu_m^{\alpha}$.
  };
  since conjugation by $g \in G$ maps $U$ to itself and hence $\Ad(g)$
  restricts to an automorphism of $\fu$.

  Since $R$ is independent of $\alpha$, so is $C$.
\end{proof}

\begin{df}
  Given a connected solvable group $G$, let $U$ be its unipotent radical and
  $\T = G / U$. Let $M = \Hom(\T, \Gm)$, and $C$ be the cone as in
  Lemma~\ref{lem:indepcon}. We shall call $(M, C)$ the \emph{preordered
  abelian group associated} to $G$.
\end{df}

\begin{thm}
  Given a connected solvable group $G$, let $(M, C)$ be the associated
  preordered abelian group. Let $\alpha$ be a section of the short exact
  sequence \eqref{eqn:sessolgp}:
  \[
	\xymatrix{
	  0 \ar[r] &
	  U \ar[r]_{\iota} &
	  G \ar[r]_{\pi} &
	  \T \ar@/_1pc/@{.>}[l]_{\alpha} \ar[r] &
	  0.
	}
  \]
  Suppose $V$ is a finite dimensional representation of
  $G$ Then the action of $\alpha(\T)$ on $V$ induces a decomposition
  \begin{equation*}
	V = \bigoplus_{\mu \in M} V_{\mu}
  \end{equation*}
  and one can define a multifiltration $F$ on $V$ as $F^m V = \bigoplus_{\mu
	\geq m} V_{\mu}$. Then the following holds.
  \begin{enumerate}
    \item \label{itm:Findpalp} The multifiltration $F$ is independent of the
      choice of $\alpha$.
    \item \label{itm:mfregsep} The multifiltration is exhaustive, separated and regular.
  \end{enumerate}
\end{thm}
\begin{proof}
  The short exact sequence $0 \to U \to G \to \T \to 0$ leads to a short exact
  sequence of the corresponding Lie algebras
  \[
    0 \to \fu \to \fg \to \ft \to 0.
  \]
  The section $\alpha : \T \to G$ induces a derived map of Lie algebras
  $d\alpha : \ft \to \fg$. Moreover, the action of $\T$ on $U$, $\Ad : \T \to
  GL(U)$, leads to the derived map
  \[
    \ad : \ft \to \gl(\fu), \quad \forall t \in \ft, u \in \fu, t \cdot u =
    \ad(d\alpha(t))(u) = [d\alpha(t), u].
  \]
  Now let $v \in V_{\mu}$ and $u \in \fu^{\alpha}_m$. We have for $t \in \ft$,
  \begin{align*}
    t \cdot (u \cdot v) 
    &= d\alpha(t) u \cdot v = [d \alpha(t), u] \cdot v + u
    d \alpha(t) \cdot v \\
    &= (d\alpha(t) \cdot u) \cdot v + u \cdot \mu(t) v
    = m(t) u
    \cdot v + \mu (t) u \cdot v \\
    &= (\mu(t) + m(t)) u \cdot v.
  \end{align*}
  Thus if $v \in V_{\mu}$ and $u \in \fu_m^{\alpha}$, then $u \cdot v \in
  V_{\mu + m}$.

  Now any $g \in \fg$ splits as $g = u + d\alpha(t)$ with $u \in \fu$ and $t \in \ft$.
  Since $d\alpha(t) \cdot v \in V_{\mu}$ for any $v \in V_{\mu}$,
  \[
    v \in V_{\mu} \implies g \cdot v \in F^{\mu}V,\ \forall g \in \fg.
  \]
  This implies that,
  \[
    v \in V_{\mu} \implies g \cdot v \in F^{\mu}V,\ \forall g \in G.
  \]

  Let $\alpha$ and $\dn{\alpha}$ be two sections, and $g\in G$ with
  $\dn{\alpha }(t)=g\alpha (t)g^{-1}$. We denote by $\bigoplus_{\mu
  }V_{\mu }$ the grading 
  of $V$ induced by $\alpha (\T)$ and by $\bigoplus _{m}\dn{V}_{\mu }$ the one induced
  by $\dn{\alpha }(T)$. Then $\dn{V}_{\mu }=g\cdot V_{\mu }$. Denote 
  on $V$ induced by them as $F_{\alpha}$ and $F_{\alpha'}$ respectively. Then
  the above discussion implies that if $v \in V_{\mu}$,
  \[
    g \cdot v \in F^{\mu}_{\alpha} V.
  \]
  Thus $F_{\alpha'}^{\mu} V \subset F_{\alpha}^{\mu} V$. By reversing the
  roles of $\alpha$ and $\alpha'$, we finish the proof of \eqref{itm:Findpalp}

  The proof of \eqref{itm:mfregsep} follows from Theorem
  \ref{thm:cridsdec} and the fact that the choice of section $\alpha $
  defines a graduation from which the filtration is defined. 
\end{proof}

\begin{ex}
  Let $U$ be a connected unipotent group and $\T$ a torus. Let $G = U \times
  \T$. Let $\alpha \colon \T \to G$ be the map $\alpha(t) = (1, t)$, where
  $1$ is the identity in $U$. It is clear that the induced action on
  $U$ and hence on the Lie algebra $\fu$ is trivial. Thus $C = R =
  \setl{0}$. In this case, for any representation $V$ of $G$, $F^m V = V_m$.
\end{ex}

\begin{ex}
  Consider the connected solvable subgroup of $\GL(3, \C)$ defined by
  \[
	G = \set{
	  \begin{pmatrix}
		\lambda & 0 & a \\
		0 & \lambda^{-1} & b \\
		0 & 0 & 1
	  \end{pmatrix}
	}{\lambda \in \C^{\ast}; a, b \in \C}.
  \]
  It is easy to check that it is a subgroup. In this case the Lie algebra of
  the unitary group, $\fu$ is two dimensional and the torus is one dimensional. It
  is clear that the subspace of $\fu$ generated by $
  \begin{pmatrix}
	0 & 0 & a \\
	0 & 0 & 0 \\
	0 & 0 & 0
  \end{pmatrix}
  $ is the eigenspace of $\chi^1$ and that generated by $
  \begin{pmatrix}
	0 & 0 & 0 \\
	0 & 0 & b \\
	0 & 0 & 0
  \end{pmatrix}
  $ is the eigenspace of $\chi^{-1}$. Thus $R = \setl{-1, 1}$ and $C = M =
  \Z$. In this case, any element of $C$ is a quasi-zero element and
  for any representation $V$, the multifiltration $F$ satisfies $F^m V
  = V$ for all $m \in M$. 
\end{ex}

\begin{ex}
  To get a nontrivial example, consider $G$ be the subgroup of all upper
  triangular matrices in $\SL(n, \C)$. In this case, the unipotent radical
  consists of matrices of the form $I + N$ where $N$ are \emph{strictly}
  upper triangular, and $\T$ consists of all diagonal matrices of
  determinant one. It is easy
  to check that in this case $M \subset \Z^{n}$ is the sublattice
  \begin{displaymath}
    (m_{1},\dots,m_{n})\in M \Longleftrightarrow \sum_{i=1}^{n}m_{n}=0
  \end{displaymath}
  and $C$ is the cone
  generated by the elements $v_{i}$, $1=1,\dots,n-1$ of $M$ of the form
  \begin{displaymath}
    v_{i}=(0,\dots,\underset{i}{1},-1,\dots,0).
  \end{displaymath}
\end{ex}

\providecommand{\bysame}{\leavevmode\hbox to3em{\hrulefill}\thinspace}
\providecommand{\MR}{\relax\ifhmode\unskip\space\fi MR }
\providecommand{\MRhref}[2]{%
  \href{http://www.ams.org/mathscinet-getitem?mr=#1}{#2}
}
\providecommand{\href}[2]{#2}

\end{document}